\documentclass[12pt,twoside]{amsart}
\usepackage{latexsym,amsfonts,amsmath,amssymb}
\usepackage{enumerate}
\usepackage{color}

\newtheorem{Th}{Theorem}[section]
\newtheorem{Rem}[Th]{Remark}
\newtheorem{Ex}[Th]{Example}
\newtheorem{Lemma}[Th]{Lemma}
\newtheorem{Def}[Th]{Definition}
\newtheorem{Prop}[Th]{Proposition}

\catcode`\@=11

\renewcommand{\section}%
   {\setcounter{equation}{0}\@startsection {section}{1}{\z@}{-3.5ex plus -1ex
  minus -.2ex}{2.3ex plus .2ex}{\Large\bf}}

\def\supp{\mathop{\rm supp}\nolimits}

\def\Im{\mathop{\rm Im}\nolimits}
\def\Wig{\mathop{\rm Wig}\nolimits}
\def\indlim{\mathop{\rm ind\,lim}}
\def\ds{\displaystyle}
\def\R{\mathbb R}
\def\C{\mathbb C}
\def\N{\mathbb N}

\newcommand{\D}{\mathcal{D}}
\newcommand{\E}{\mathcal{E}}
\newcommand{\F}{\mathcal{F}}

\newcommand{\Sch}{\mathcal{S}}

\newcommand{\afrac}[2]{\genfrac{}{}{0pt}{1}{#1}{#2}}

\newcommand{\beqsn}{\arraycolsep1.5pt\begin{eqnarray*}}
\newcommand{\eeqsn}{\end{eqnarray*}\arraycolsep5pt}
\newcommand{\beqs}{\arraycolsep1.5pt\begin{eqnarray}}
\newcommand{\eeqs}{\end{eqnarray}\arraycolsep5pt}

\title{Regularity of partial differential operators in ultradifferentiable
  spaces and Wigner type transforms}

\author[Boiti]{Chiara Boiti}
\address{
Dipartimento di Matematica e Informatica \\Universit\`a di Ferrara\\
Via Ma\-chia\-vel\-li n.~30\\
I-44121 Ferrara\\
Italy}
\email{chiara.boiti@unife.it}

\author[Jornet]{David Jornet}
\address{
Instituto Universitario de Matem\'atica Pura y Aplicada IUMPA\\
Universitat Po\-li\-t\`ecni\-ca de Val\`encia\\
Camino de Vera, s/n\\
E-46071 Valencia\\
Spain}
\email{djornet@mat.upv.es}

\author[Oliaro]{Alessandro Oliaro}
\address{Dipartimento di Matematica\\ Universit\`a di Torino\\
  Via Carlo Alberto n. 10\\ I-10123 Torino\\ Italy}
  \email{alessandro.oliaro@unito.it}

\begin{document}


\begin{abstract}
  We study the behaviour of linear partial differential operators with
  polynomial coefficients via
  a Wigner type transform.
  In particular, we obtain some results of regularity
  in the Schwartz space $\Sch$ and in the space $\Sch_\omega$ as introduced by Bj\"orck for
  weight functions
  $\omega$. Several examples are discussed in this new setting.
\end{abstract}

\maketitle
% computing hour and minutes
%  \newcount\minutes
%  \newcount\hour
%  \minutes=\time
%  \divide\minutes by 60
%  \hour=\minutes
%  \minutes=\time
%  \multiply \hour by 60
%  \advance \minutes by -\hour
%  \divide \hour by 60
%  \markboth{\today\space\number\hour :\number\minutes}%
%  {\today\space\number\hour :\number\minutes}
\markboth{\sc Regularity of partial differential operators \ldots}
{\sc C.~Boiti, D.~Jornet and A.~Oliaro}

  \section{Introduction}

In this paper we are concerned with the regularity of linear partial
differential operators with polynomial coefficients. This problem was
introduced by Shubin \cite{S}, who says that a linear operator
$A:\mathcal{S}^\prime\to\mathcal{S}^\prime$ is regular if the conditions
$u\in\mathcal{S}^\prime$, $Au\in\mathcal{S}$ imply that $u\in\mathcal{S}$.
In \cite[Chapter IV]{S} global pseudodifferential operators on $\mathbb{R}^n$
are studied, giving a notion of (global) hypoellipticity (see formula (\ref{g-hypoelliptic})), that implies the
above mentioned regularity in Schwartz spaces. Such global pseudodifferential
operators are defined by treating in the same way variables and covariables,
and have as basic examples linear partial differential operators with polynomial
coefficients. The global hypoellipticity, on the other hand, is far from
being a necessary condition for the regularity of an operator; some results
have been obtained in this direction, we refer in particular to \cite{W} who
proved the regularity of the Twisted Laplacian (a non hypoelliptic operator
in two variables), and to \cite{NR}, who gave a characterization of the
regularity of ordinary differential operators in the case when the roots
of the corresponding Weyl symbol are suitably separated at infinity.
Moreover, in \cite{BO} a class of non hypoelliptic regular partial
differential operators with polynomial coefficients have been found,
by using a technique related to transformations of Wigner type; such
class includes as a particular case the Twisted Laplacian. The idea to use quadratic transformations for the study of general properties of partial differential equations (that underlies \cite{BO}, as well as the present paper) goes back to some works related to engineering applications, cf. \cite{CG}, \cite{GC}, where the main aim is to understand the Wigner transform of the solution of a partial differential equation without finding the solution itself; the ideas of \cite{CG}, \cite{GC} are developed and organically presented in \cite{C}. In the present paper we study the regularity of linear partial differential operators, in the spirit of \cite{BO}, developing the research in two directions; first, we
consider a general representation in the Cohen class, defined as
$$
Q[w]:=\sigma * \Wig[w]
$$
for a kernel $\sigma\in\mathcal{S}^\prime$, where $\Wig[w]$ is the Wigner
transform, defined as
$$
\Wig[w](x,y):=\int e^{-ity} w\left( x+\frac{1}{2}t,x-\frac{1}{2}t\right)\,dt.
$$
The idea is that a linear partial differential operator $B$ with polynomial coefficients is
transformed into another one by a formula of the kind
$$
Q[Bw]=\tilde{B} Q[w];
$$
moreover, under suitable hypotheses on the kernel $\sigma$, the
regularity is preserved by such transformation, and if we start from a
global hypoelliptic operator $B$ we find in general a non-global
hypoelliptic operator $\tilde{B}$. Then, we can construct a large class
of partial differential operators that are regular but not globally
hypoelliptic. We also study regularity and the results just mentioned for the class $\mathcal{S}_\omega$ for a weight function $\omega$, as introduced by Bj\"orck~\cite{B} (see also \cite{F} for non subadditive weight functions), which gives a large scale of examples, working in particular for Gevrey weight functions.
This requires a preliminary study of the Schwartz ultradifferentiable
space $\mathcal{S}_\omega$ and of
the Cohen class representation $Q$
in $\mathcal{S}_\omega$
and $\mathcal{S}^\prime_\omega$.
In particular, we give a characterization of the spaces $\Sch_\omega$,
improving a result of \cite{CKL},  introducing a new kind of seminorms in the spirit of the spaces of ultradifferentiable functions introduced by Braun, Meise and Taylor~\cite{BMT} (compare with Langenbruch~\cite{L}).

The examples that we can construct with our
technique are quite general, we mention here some cases. We show for example
that, if $b$ is a polynomial in one variable that never vanishes, and
$P(D_x,D_y)$ is an arbitrary partial differential operator with constant
real coefficients, then the operator
$$
b(x+P(D_x,D_y))
$$
in $\mathbb{R}^2$ is regular in the sense of Shubin and in the sense of
ultradifferentiable classes $\mathcal{S}_\omega$. The same is true for
the operator in two variables
$$
(x-D_y+Q(D_x))^2+(y+R(D_y))^2,
$$
for arbitrary ordinary differential operators $Q(D_x)$ and $R(D_y)$
with constant real coefficients. Observe in particular that the regularity
here does not depend on the higher order terms, since the operators
$P$, $Q$, $R$ can have arbitrary order.

The paper is organized as follows. Section \ref{sec2} is devoted to the
study of some properties of the Wigner transform in $\mathcal{S}$, that
we use in the following; in Sections \ref{sec3} and \ref{secSomega} we
study the global regularity through Cohen class representations in
$\mathcal{S}$ and $\mathcal{S}_\omega$, respectively; finally, in the
last section we analyze some examples. The results are proved in the
case of dimension $2$, for sake of simplicity, but they could
easily be generalized to higher even dimension.

\section{Some properties of the Wigner transform on $\Sch$}
\label{sec2}

Let us define, following \cite{BO}, the Wigner-like transform of a function
$w\in\Sch(\R^2)$, by
\beqs
\label{1}
\Wig[w](x,y):=\int e^{-ity}w\Big(x+\frac12 t,x-\frac12 t\Big)dt.
\eeqs
In this way
\beqsn
&&\Wig:\ \Sch\to\Sch\\
&&\Wig:\ \Sch'\to\Sch'
\eeqsn
is invertible, since it is the composition of a linear invertible change
of variables and a partial Fourier transform, i.e.
\beqs
\label{wigner-fourier}
\Wig[w](x,y)=\F_t\left(Tw(x,t)\right)(x,y)
\eeqs
with
\beqs
\label{wigner-comp-op}
Tw(x,t)=w\Big(x+\frac12 t,x-\frac12 t\Big),
\qquad
\F(f)(y)=\F_t(f(t))(y)=\int e^{-ity}f(t)dt.
\eeqs

Denote, as in \cite{BO},
\beqsn
&&M_1w(x,y)=xw(x,y),\qquad M_2w(x,y)=yw(x,y),\\
&&D_1w(x,y)=D_xw(x,y),\qquad D_2w(x,y)=D_yw(x,y),
\eeqsn
with $D_x=-i\partial_x$, $D_y=-i\partial_y$, and recall,
from \cite{BO}, the following properties:
\beqs
\label{3}
&&D_1\Wig[w]=\Wig[(D_1+D_2)w]\\
\label{4}
&&D_2\Wig[w]=\Wig[(M_2-M_1)w]\\
\label{5}
&&M_1\Wig[w]=\Wig\left[\frac12(M_2+M_1)w\right]\\
\label{6}
&&M_2\Wig[w]=\Wig\left[\frac12(D_1-D_2)w\right].
\eeqs

More generally,
let $P(D_1,D_2)=\sum_{|(h,k)|\leq m}a_{hk}D_x^h D_y^k$ be a linear partial
differential operator with constant coefficients and denote by
\beqs
\label{2}
P(D_1+D_2,M_2-M_1)=\sum_{|(h,k)|\leq m}a_{hk}(D_1+D_2)^h(M_2-M_1)^k,
\eeqs
which is a linear partial differential operator with polynomial coefficients.

Note that
\beqs
\label{11}
(D_1+D_2)^h(M_2-M_1)^k=(M_2-M_1)^k(D_1+D_2)^h
\eeqs
since
\beqsn
(D_1+D_2)(M_2-M_1)w=&&D_1M_2w-D_1M_1w+D_2M_2w-D_2M_1w\\
=&&M_2D_1w{+i}w-M_1D_1w{-i}w+M_2D_2w-M_1D_2w\\
=&&(M_2-M_1)D_1w+(M_2-M_1)D_2{w}\\
=&&(M_2-M_1)(D_1+D_2)w.
\eeqsn

We have the following
\begin{Lemma}
\label{lemma1}
Let $P(D_x,D_y)$ be a linear partial differential operator with constant
coefficients.
Then, for every $w\in\Sch(\R^2)$,
\beqs
\label{12}
P(D_1,D_2)\Wig[w](x,y)=\Wig[P(D_1+D_2,M_2-M_1)w](x,y).
\eeqs
\end{Lemma}

\begin{proof}
Let us prove by induction on $h\in\N_0:=\N\cup\{0\}$ that
\beqs
\label{7}
D_1^hD_2^k\Wig[w]=\Wig[(D_1+D_2)^h(M_2-M_1)^kw]
\eeqs
for all $k\in\N_0$.

Indeed, for $h=0$ by \eqref{4} we have that
\beqsn
D_2^k\Wig[w]=\Wig[(M_2-M_1)^kw].
\eeqsn
Let us assume \eqref{7} to be true for $h$ and prove it for $h+1$. By
the inductive assumption \eqref{7} and \eqref{3}:
\beqsn
D_1^{h+1}D_2^k\Wig[w]=&&D_1\Wig[(D_1+D_2)^h(M_2-M_1)^kw]\\
=&&\Wig[(D_1+D_2)^{h+1}(M_2-M_1)^kw].
\eeqsn
Moreover, since $\Wig[w]\in C^2(\R^2)$ we have that
\beqs
\label{9}
D_1^hD_2^k\Wig[w]=D_2^kD_1^h\Wig[w]=\Wig[(D_1+D_2)^h(M_2-M_1)^kw].
\eeqs

The thesis then follows from \eqref{9} and the definition of $P$.
\end{proof}

Analogous formulas hold for linear partial differential operators with
polynomial coefficients:
\begin{Prop}
\label{cor1}
Let $P(x,y,D_x,D_y)$ be a linear partial differential operator with
polynomial coefficients. Then, for all $w\in\Sch(\R^2)$, the
following formula holds:
\beqsn
P(M_1,M_2,D_1,D_2)\Wig[w]=
\Wig\left[P\left(\frac12(M_2+M_1),\frac12(D_1-D_2),D_1+D_2,M_2-M_1\right)
w\right].
\eeqsn
\end{Prop}

\begin{proof}
From Lemma~\ref{lemma1} and \eqref{5},\eqref{6} we have that
\beqsn
M_1^mM_2^nD_1^hD_2^k\Wig[w]=\Wig\left[\frac{1}{2^{n+m}}(M_2+M_1)^m
(D_1-D_2)^n(M_2-M_1)^k(D_1+D_2)^hw\right]
\eeqsn
and hence the thesis, since $M_1^mM_2^n=M_2^nM_1^m$.

Note that, analogously to \eqref{11}, we have that
\beqsn
(D_1-D_2)^n(M_2+M_1)^m=(M_2+M_1)^m(D_1-D_2)^n.
\eeqsn
\end{proof}

\section{Properties and regularity of time-frequency representations in the
Cohen's class with kernel in $\Sch'$}\label{sec3}

Let us now consider a time-frequency representation $Q[w]$ in the
Cohen's class, i.e. of the form
\beqsn
Q[w]:=\sigma*\Wig[w]
\eeqsn
for $w\in\Sch(\R^2)$ and $\sigma\in\Sch'(\R^2)$.

By \eqref{3}, \eqref{4}  we have that
\beqs
\label{13}
&&D_1Q[w]=\sigma*D_1\Wig[w]=Q[(D_1+D_2)w]\\
\label{14}
&&D_2Q[w]=\sigma*D_2\Wig[w]=Q[(M_{2}-M_1)w].
\eeqs
Moreover, let us prove that
\beqs
\label{15}
&&M_1Q[w]=Q\left[\frac12(M_2+M_1)w\right]+(M_1\sigma)*\Wig[w]\\
\label{16}
&&M_2Q[w]=Q\left[\frac12(D_1-D_2)w\right]+(M_2\sigma)*\Wig[w].
\eeqs
Indeed, from \eqref{5} and \eqref{6}:
\beqsn
M_1Q[w](x,y)=&&\int x\sigma(\alpha,\beta)\Wig[w](x-\alpha,y-\beta)d\alpha
d\beta\\
=&&\int\sigma(\alpha,\beta)(x-\alpha)\Wig[w](x-\alpha,y-\beta)d\alpha d\beta\\
&&+\int\alpha\sigma(\alpha,\beta)\Wig[w](x-\alpha,y-\beta)d\alpha d\beta\\
=&&\sigma*(M_1\Wig[w])+(M_1\sigma)*\Wig[w]\\
=&&\sigma*\Wig\left[\frac12(M_2+M_1)w\right]+(M_1\sigma)*\Wig[w]
\eeqsn
and analogously
\beqsn
M_2Q[w](x,y)=&&\int y\sigma(\alpha,\beta)\Wig[w](x-\alpha,y-\beta)d\alpha
d\beta\\
=&&\sigma*(M_2\Wig[w])+(M_2\sigma)*\Wig[w]\\
=&&\sigma*\Wig\left[\frac12(D_1-D_2)w\right]+(M_2\sigma)*\Wig[w],
\eeqsn
where the integrals are intended as the action of the distribution
$\sigma$ when $\sigma$ is not a function.

In order to write also \eqref{15} and \eqref{16} in terms of $Q$ applied
to some $\tilde{P}(M_1,M_2,D_1,D_2)w$, for a linear partial differential
operator $\tilde{P}$ with polynomial coefficients, we now choose
$\sigma(\alpha,\beta)$ so that
\beqs
\label{17}
\begin{cases}
M_1\sigma(\alpha,\beta)=\alpha\sigma(\alpha,\beta)=P_1(D_\alpha,D_\beta)
\sigma(\alpha,\beta)\cr
M_2\sigma(\alpha,\beta)=\beta\sigma(\alpha,\beta)=P_2(D_\alpha,D_\beta)
\sigma(\alpha,\beta)
\end{cases}
\eeqs
for some linear partial differential operators $P_1,P_2$ with
constant coefficients.

Let us solve \eqref{17} by Fourier transform:
\beqs
\label{18}
\begin{cases}
P_1(\xi,\eta)\widehat{\sigma}(\xi,\eta)=i\partial_\xi
\widehat{\sigma}(\xi,\eta)\cr
P_2(\xi,\eta)\widehat{\sigma}(\xi,\eta)=i\partial_\eta
\widehat{\sigma}(\xi,\eta).
\end{cases}
\eeqs

By simple computations, chosen any given real valued polynomial
$P(\xi,\eta)\in\R[\xi,\eta]$,
we can thus set
\beqsn
P_1(\xi,\eta)=\partial_\xi P(\xi,\eta),\qquad
P_2(\xi,\eta)=\partial_\eta P(\xi,\eta)
\eeqsn
and obtain that
\beqs
\label{19}
\widehat{\sigma}(\xi,\eta)=e^{-iP(\xi,\eta)}\in\Sch'(\R^2)
\eeqs
solves \eqref{18} (note that $|\widehat{\sigma}|=1$).
Since the Fourier transform $\F:\ \Sch'\to\Sch'$ is invertible, we have
that
\beqs
\label{20}
\sigma(\alpha,\beta)=\F^{-1}\left(e^{-iP(\xi,\eta)}\right)\in\Sch'(\R^2)
\eeqs
solves \eqref{17}.

For such a choice of $\sigma$, substituting in \eqref{15},
by Lemma~\ref{lemma1} we get:
\beqs
\label{28}
M_1Q[w]=&&Q\left[\frac12(M_2+M_1)w\right]+P_1(D_1,D_2)\sigma*\Wig[w]\\
\nonumber
=&&Q\left[\frac12(M_2+M_1)w\right]+\sigma*P_1(D_1,D_2)\Wig[w]\\
\nonumber
=&&Q\left[\frac12(M_2+M_1)w\right]+\sigma*\Wig[P_1(D_1+D_2,M_2-M_1)w]\\
\nonumber
=&&Q\left[\left(\frac12(M_2+M_1)+P_1(D_1+D_2,M_2-M_1)\right)w\right]\\
\label{21}
=&&Q\left[\left(\frac12(M_2+M_1)+(iD_1P)(D_1+D_2,M_2-M_1)\right)w\right].
\eeqs
Analogously, from \eqref{16}:
\beqs
\label{29}
M_2Q[w]=&&Q\left[\frac12(D_1-D_2)w\right]+P_2(D_1,D_2)\sigma*\Wig[w]\\
\nonumber
=&&Q\left[\left(\frac12(D_1-D_2)+P_2(D_1+D_2,M_2-M_1)\right)w\right]\\
\label{22}
=&&Q\left[\left(\frac12(D_1-D_2)+(iD_2P)(D_1+D_2,M_2-M_1)\right)w\right].
\eeqs

Iterating this procedure we get the following:
\begin{Th}
\label{th1}
Let $B(x,y,D_x,D_y)$ be a linear partial differential operator with
polynomial coefficients and let $\sigma=\F^{-1}(e^{-iP(\xi,\eta)})\in\Sch'(\R^2)$
for some $P\in\R[\xi,\eta]$. Then, for every $w\in\Sch(\R^2)$, the
time-frequency representation $Q[w]=\sigma*\Wig[w]$ satisfies:
\beqs
\label{23}
B(M_1,M_2,D_1,D_2)Q[w]=Q[\bar{B}(M_1,M_2,D_1,D_2)w],
\eeqs
where $\bar{B}$  is the linear partial differential operator with
polynomial coefficients defined by
\beqs
\label{24}
\bar{B}(M_1,M_2,D_1,D_2):=B\left(\frac{M_2+M_1}{2}+P^*_1,
\frac{D_1-D_2}{2}+P^*_2,D_1+D_2,M_2-M_1\right),
\eeqs
with
\beqsn
P^*_1=(iD_1P)(D_1+D_2,M_2-M_1),\qquad P^*_2=(iD_2P)(D_1+D_2,M_2-M_1).
\eeqsn
\end{Th}

\begin{proof}
By \eqref{13}, \eqref{14} and \eqref{11} we immediately get
\beqs
\label{26}
D_1^hD_2^kQ[w]=D_2^kD_1^hQ[w]=Q[(D_1+D_2)^h(M_2-M_1)^kw].
\eeqs

Let us prove by induction on $m\in\N_0$ that
\beqs
\label{25}
M_1^mQ[w]=Q\left[\left(\frac{M_2+M_1}{2}+P^*_1\right)^mw\right].
\eeqs
Indeed, for $m=0$ there is nothing to prove.
Let us assume \eqref{25} true for $m$ and prove it for $m+1$.
By the inductive assumption \eqref{25} and by \eqref{21} we have that
\beqsn
M_1^{m+1}Q[w]=&&M_1Q\left[\left(\frac{M_2+M_1}{2}+P^*_1\right)^mw\right]\\
=&&Q\left[\left(\frac{M_2+M_1}{2}+P_1^*\right)
\left(\frac{M_2+M_1}{2}+P^*_1\right)^mw\right]\\
=&&Q\left[\left(\frac{M_2+M_1}{2}+P^*_1\right)^{m+1}w\right].
\eeqsn

Analogously, by \eqref{22} we can prove by induction on $n\in\N_0$ that
\beqs
\label{27}
M_2^nQ[w]=Q\left[\left(\frac{D_1-D_2}{2}+P_2^*\right)^nw\right].
\eeqs
The thesis then follows from \eqref{26}, \eqref{25} and \eqref{27}.
\end{proof}

Reciprocally, we have the following:
\begin{Th}
\label{th2}
Let $B(x,y,D_x,D_y)$ be a linear partial differential operator with
polynomial coefficients and let $\sigma=\F^{-1}(e^{-iP(\xi,\eta)})\in\Sch'(\R^2)$
for some $P\in\R[\xi,\eta]$. Then, for every $w\in\Sch(\R^2)$, the
time-frequency representation $Q[w]=\sigma*\Wig[w]$ satisfies:
\beqs
\label{30}
Q[B(M_1,M_2,D_1,D_2)w]=\tilde{B}(M_1,M_2,D_1,D_2)Q[w],
\eeqs
where $\tilde{B}$ is the linear partial differential operator with
polynomial coefficients defined by
\beqs
\label{31}
&&\tilde{B}(M_1,M_2,D_1,D_2)\\
\nonumber
=&&B\left(M_1-\frac12D_2-P_1,M_1+\frac12D_2-P_1,
\frac12D_1+M_2-P_2,\frac12D_1-M_2+P_2\right)
\eeqs
with
\beqs
\label{38}
P_1=(iD_1P)(D_1,D_2),\qquad P_2=(iD_2P)(D_1,D_2).
\eeqs
\end{Th}

\begin{proof}
From \eqref{13}, \eqref{14}, \eqref{28} and \eqref{29} we have:
\beqs
\label{a}
&&D_1Q[w]=Q[D_1w]+Q[D_2w]\\
\label{b}
&&D_2Q[w]=Q[M_2w]-Q[M_1w]\\
\label{c}
&&M_1Q[w]=\frac12Q[M_2w]+\frac12Q[M_1w]+P_1Q[w]\\
\label{d}
&&M_2Q[w]=\frac12Q[D_1w]-\frac12Q[D_2w]+P_2Q[w].
\eeqs
Therefore, from \eqref{a} and \eqref{d}:
\beqsn
&&Q[D_1w]=\left(\frac12D_1+M_2-P_2\right)Q[w]\\
&&Q[D_2w]=\left(\frac12D_1-M_2+P_2\right)Q[w];
\eeqsn
from \eqref{b} and \eqref{c}:
\beqsn
&&Q[M_1w]=\left(M_1-\frac12D_2-P_1\right)Q[w]\\
&&Q[M_2w]=\left(M_1+\frac12D_2-P_1\right)Q[w].
\eeqsn
Iterating:
\beqs
\nonumber
Q[M_1^mM_2^nD_1^hD_2^kw]
=&&\left(M_1-\frac12D_2-P_1\right)^m\left(M_1+\frac12D_2-P_1\right)^n\\
\label{32}
&&\cdot\left(\frac12D_1+M_2-P_2\right)^h\left(\frac12D_1-M_2+P_2\right)^kQ[w].
\eeqs

Let us remark that $\left(M_1-\frac12D_2-P_1(D_1,D_2)\right)$ and
$\left(M_1+\frac12D_2-P_1(D_1,D_2)\right)$ commute, as also
$\left(\frac12D_1+M_2-P_2(D_1,D_2)\right)$ and
$\left(\frac12D_1-M_2+P_2(D_1,D_2)\right)$. On the other hand,
$M_1$ and $M_2$ commute and also $D_1D_2w=D_2D_1w$ since $w\in
C^\infty(\R^2)$.
The thesis follows therefore from \eqref{32}.
\end{proof}

In order to prove further properties of $Q$, let us define the space
$C^\infty_p$ of $C^\infty$ functions with polynomial growth:
\beqsn
C^\infty_p{(\mathbb{R}^n)}:=\{\varphi\in C^\infty(\R^n):\ \exists N\in\N,\,c>0\
\mbox{s.t.}\
|\partial^\gamma\varphi(x)|\leq c(1+|x|^2)^N\ \forall x\in\R^n,\,\gamma\in
\N_0^n\}.
\eeqsn
The last space is included in the space of multipliers $\mathcal{O}_M(\R^n)$ of the space $\mathcal{S}(\R^n)$, i.e., the space of smooth functions $F$ such that $F\, \mathcal{S}(\R^n)\subset \mathcal{S}(\R^n)$. Indeed, it is known that $F\in  \mathcal{O}_M(\R^n)$ if and only if for each $k\in\N$ there is $C>0$ and $j\in\N$ such that $|F^{(\alpha)}(x)|\le C(1+|x|)^j$ for all multi-index $\alpha$ with $|\alpha|\le k$. Then, the next lemma is obvious \cite{horvath,schwartz}.

\begin{Lemma}
\label{lemma4}
Let $\varphi\in C^\infty_p(\R^n)$. If $u\in\Sch(\R^n)$, then
$\varphi u\in\Sch(\R^n)$; if $w\in\Sch'(\R^n)$, then
$\varphi w\in\Sch'(\R^n)$.

\end{Lemma}

%\begin{proof}
%For all $\alpha,\beta\in\N^n_0$ we have:
%\beqsn
%|x^\alpha\partial^\beta(\varphi u)(x)|\leq&&\sum_{\gamma\leq\beta}C_{\gamma,\beta}
%|x^\alpha\partial^\gamma\varphi(x)\cdot\partial^{\beta-\gamma}u(x)|\\
%=&&\sum_{\gamma\leq\beta}C_{\gamma,\beta}
%\frac{|\partial^\gamma\varphi(x)|}{(1+|x|^2)^N}\cdot
%|(1+|x|^2)^Nx^\alpha\partial^{\beta-\gamma}u(x)|\leq C_{\alpha,\beta}
%\eeqsn
%for some $C_{\alpha,\beta}>0$, since $|\partial^\gamma\varphi(x)|/(1+|x|^2)^N$ is
%bounded by definition of the space $C^\infty_p$ and
%$(1+|x|^2)^Nx^\alpha\partial^{\beta-\gamma} u(x)$ is bounded for $u\in\Sch$.
%Therefore $\varphi u\in\Sch$. {Observe moreover that, by the same computations,
%  if $\psi_j\in\Sch(\R)$, $\psi_j\to 0$ in $\Sch$, then $\varphi\psi_j\to 0$ in
%  $\Sch$.}
%
%The last statement then also follows, since
%\beqsn
%\langle\varphi w,\psi\rangle=\langle w,\varphi\psi\rangle
%\qquad\forall\psi\in\Sch,
%\eeqsn
%    {so $\varphi w$ is well defined on $\Sch$, and it is continuous since for
%      every sequence $\psi_j\to 0$ in $\Sch$ we have $\langle\varphi w,\psi_j
%      \rangle = \langle w,\varphi\psi_j\rangle\to 0$ from the previous point.}
%\end{proof}

We recall the notion of regularity from \cite{S}:
\begin{Def}
\label{def-reg-S}
A linear operator $A$ on $\Sch'(\R^n)$ is {\em regular} if
\beqsn
Au\in\Sch(\R^n)\quad \Rightarrow\quad
u\in\Sch(\R^n),\qquad\forall u\in\Sch'(\R^n).
\eeqsn
\end{Def}

We have the following:
\begin{Lemma}
\label{lemma2}
For $\sigma=\F^{-1}(e^{-iP(\xi,\eta)})\in\Sch'(\R^2)$
with $P\in\R[\xi,\eta]$ and $Q[w]=\sigma*\Wig[w]$, we have that:
\begin{itemize}
\item[(i)]
$Q:\ \Sch'\to\Sch'$ is invertible;
\item[(ii)]
$Q$ is {\em regular};
\item[(iii)]
$Q:\ \Sch\to\Sch$.
\end{itemize}
\end{Lemma}

\begin{proof}
Let us first prove that if $w\in\Sch'$ then $Q[w]$ is a well defined element of
$\Sch'$. As a matter of fact,
$
\widehat{Q[w]}=\widehat{\sigma}\cdot\widehat{\Wig[w]}\in\Sch'
$
because of Lemma~\ref{lemma4}, since $\widehat{\sigma}\in C^\infty_p$
and $\widehat{\Wig[w]}\in\Sch'$ for $w\in\Sch'$.
Then $Q[w]$ is well defined as
$\F^{-1}\left(\widehat{\sigma}\cdot\widehat{\Wig[w]}\right)\in\Sch'$.

The injectivity of $Q:\ \Sch'\to\Sch'$ is trivial. To prove the surjectivity,
take $w\in\Sch'$. Then $\widehat{w}\in\Sch'$ and, by Lemma~\ref{lemma4},
also $\widehat{w}/\widehat{\sigma}\in\Sch'$ since
$1/\widehat{\sigma}\in C^\infty_p$. By the surjectivity of the Fourier
transform there exists $\psi\in\Sch'$ such that
$\widehat{w}/\widehat{\sigma}=\widehat{\psi}$. By the surjectivity of
the Wigner transform, $\psi=\Wig[u]$ for some $u\in\Sch'$ and therefore
\beqsn
\widehat{w}=\widehat{\sigma}\widehat{\psi}=\widehat{\sigma}\cdot
\widehat{\Wig[u]}=\!\!\!\!\!\widehat{\quad\sigma*\Wig[u]}=\widehat{Q[u]}
\eeqsn
and by the injectivity of the Fourier transform $w=Q[u]$. This proves $(i)$.

To prove condition $(ii)$, assume that $Q[w]\in\Sch$ for some $w\in\Sch'$.
From $\widehat{Q[w]}=\widehat{\sigma}\cdot\widehat{\Wig[w]}\in\Sch$ we thus
have that
$\widehat{\Wig[w]}\in\Sch$ since $|\widehat{\sigma}|=1$.
Therefore $\Wig[w]\in\Sch$ and hence $w\in\Sch$. This proves that $Q$ is
regular.

Finally, to prove $(iii)$ let us remark that, for $w\in\Sch$,
\beqsn
Q[w]=\sigma*\Wig[w]=\F^{-1}\left(\widehat{\sigma}\cdot
\widehat{\Wig[w]}\right)\in\Sch
\eeqsn
because of Lemma~\ref{lemma4}, since $\widehat{\sigma}\in C^\infty_p$ and
$\widehat{\Wig[w]}\in\Sch$ for $w\in\Sch$.
\end{proof}

\begin{Th}
\label{th3}
Let $B(x,y,D_x,D_y)$ be a linear partial differential operator with
polynomial coefficients and let $\sigma=\F^{-1}(e^{-iP(\xi,\eta)})\in\Sch'(\R^2)$
for some $P\in\R[\xi,\eta]$.

If $B$ is regular and $\bar{B}$ is defined by \eqref{24}, then also
$\bar{B}$ is regular.
\end{Th}

\begin{proof}
Let us assume that $\bar{B}w\in\Sch$ for $w\in\Sch'$ and prove that
$w\in\Sch$.

Indeed, $Q[w]=\sigma*\Wig[w]\in\Sch'$ by Lemma~\ref{lemma2}~$(i)$ and,
by Theorem~\ref{th1} and Lemma~\ref{lemma2}~$(iii)$, we get that $BQ[w]=Q[\bar{B}w]\in\Sch$.
Since $B$ is regular by assumption, we have that $Q[w]\in\Sch$ and hence
$w\in\Sch$ by the regularity of $Q$ from Lemma~\ref{lemma2}~$(ii)$.
\end{proof}

\begin{Th}
\label{th4}
Let $B(x,y,D_x,D_y)$ be a linear partial differential operator with
polynomial coefficients and let $\sigma=\F^{-1}(e^{-iP(\xi,\eta)})\in\Sch'(\R^2)$
for some $P\in\R[\xi,\eta]$.

If $B$ is regular and $\tilde{B}$ is defined by \eqref{31}, then also
$\tilde{B}$ is regular.
\end{Th}

\begin{proof}
Let us assume $\tilde{B}u\in\Sch$ for $u\in\Sch'$ and prove that
$u\in\Sch$.

Indeed, by the surjectivity of $Q[w]=\sigma*\Wig[w]$
(cf. Lemma~\ref{lemma2}~$(i)$),
there exists $w\in\Sch'$ such that
$u=Q[w]$ and hence, from Theorem~\ref{th2},
\beqsn
Q[Bw]=\tilde{B}Q[w]=\tilde{B}u\in\Sch.
\eeqsn
By the regularity of $Q$ (cf. Lemma~\ref{lemma2}~$(ii)$) we have that
$Bw\in\Sch$ and hence $w\in\Sch$ by the regularity of $B$.
Then also $u=Q[w]\in\Sch$ by Lemma~\ref{lemma2}~$(iii)$.
\end{proof}

\vspace{3mm}
Let us now consider
\beqs
\label{34}
\widehat{\sigma}_1(\xi,\eta)=q(\xi,\eta)\widehat{\sigma}(\xi,\eta)
=q(\xi,\eta)e^{-iP(\xi,\eta)},
\eeqs
where $\sigma$ is defined by \eqref{20} for $P(\xi,\eta)\in\R[\xi,\eta]$
and $q(\xi,\eta)\in\C[\xi,\eta]$ is a polynomial
that never vanishes on $\R^2$.
Then
\beqsn
\sigma_1(x,y)=q(D_x,D_y)\sigma(x,y)
\eeqsn
and, by Lemma~\ref{lemma1}:
\beqs
\nonumber
Q^{(\sigma_1)}[w]:=&&\sigma_1*\Wig[w]=(q(D_1,D_2)\sigma)*\Wig[w]\\
\nonumber
=&&\sigma*(q(D_1,D_2)\Wig[w])\\
\label{36}
=&&\sigma*\Wig[q(D_1+D_2,M_2-M_1)w]=Q^{(\sigma)}[Aw],
\eeqs
for $A(M_1,M_2,D_1,D_2):=q(D_1+D_2,M_2-M_1)$.

\begin{Prop}
\label{prop3}
Let $\sigma=\F^{-1}(e^{-iP(\xi,\eta)})\in\Sch'(\R^2)$
with $P\in\R[\xi,\eta]$, $\sigma_1=q(D_1,D_2)\sigma$ for a polynomial
$q(\xi,\eta)$ that never vanishes on $\R^2$, and set
$Q^{(\sigma_1)}[w]=\sigma_1*\Wig[w]$.
 Then:
\begin{itemize}
\item[(i)]
$Q^{(\sigma_1)}:\ \Sch'\to\Sch'$ is invertible;
\item[(ii)]
$Q^{(\sigma_1)}$ is {\em regular};
\item[(iii)]
$Q^{(\sigma_1)}:\ \Sch\to\Sch$.
\end{itemize}
\end{Prop}

\begin{proof}
The proof is analogous to that of Lemma~\ref{lemma2},
since $\widehat{\sigma}_1(\xi,\eta)=q(\xi,\eta)\widehat{\sigma}(\xi,\eta)$
and $q(\xi,\eta)$ never vanishes.
\end{proof}

\begin{Th}
\label{cor2}
Let $B(x,y,D_x,D_y)$ be a linear partial differential operator with polynomial
coefficients.
Let $\sigma=\F^{-1}(e^{-iP(\xi,\eta)})\in\Sch'(\R^2)$ for some
$P\in\R[\xi,\eta]$ and $\sigma_1=q(D_1,D_2)\sigma$
for some $q\in\C[\xi,\eta]$ never vanishing on $\R^2$. Then $Q^{(\sigma)}[w]=
\sigma*\Wig[w]$ and $Q^{(\sigma_1)}[w]=\sigma_1*\Wig[w]$ satisfy, for
$w\in\Sch'(\R^2)$,
\beqs
\label{35}
Q^{(\sigma_1)}[Bw]=\widetilde{AB}\,Q^{(\sigma)}[w],
\eeqs
where $A$ is the operator defined by $A(M_1,M_2,D_1,D_2)=q(D_1+D_2,M_2-M_1)$,
 and $\widetilde{AB}$ is
obtained from $AB$ as in \eqref{31}.
Moreover, $B$ is regular if and only if $\widetilde{AB}$ is regular.
\end{Th}

\begin{proof}
The equality \eqref{35} follows from \eqref{36} and Theorem~\ref{th2}.

Assume, now, that $B$ is regular. We prove that $\widetilde{AB}$ is regular.
Let $\widetilde{AB}u\in\Sch$ for some $u\in\Sch'$.
Since $Q^{(\sigma)}$ is surjective because of Lemma~\ref{lemma2}, there
exists $w\in\Sch'$ such that $u=Q^{(\sigma)}[w]$.
By \eqref{35}
\beqs
\label{37}
Q^{(\sigma_1)}[Bw]=\widetilde{AB}\,Q^{(\sigma)}[w]=\widetilde{AB}u\in\Sch.
\eeqs
But $Q^{(\sigma_1)}$ is regular by Proposition~\ref{prop3}~$(ii)$ and hence
$Bw\in\Sch$.
Therefore $w\in\Sch$ since $B$ is regular by assumption.
Then also $u=Q^{(\sigma)}[w]\in\Sch$ by Lemma~\ref{lemma2}~$(iii)$.

Reciprocally, let $\widetilde{AB}$ be regular. We prove that $B$ is regular.
Let $Bw\in\Sch$ with $w\in\Sch'$. Since $Q^{(\sigma_1)}:\ \Sch\to\Sch$
by Proposition~\ref{prop3}~$(iii)$, then $\widetilde{AB}Q^{(\sigma)}[w]
=Q^{(\sigma_1)}[Bw]\in\Sch$ and hence $Q^{(\sigma)}[w]\in\Sch$ by
the regularity of $\widetilde{AB}$.
But $Q^{(\sigma)}$ is regular by Lemma~\ref{lemma2}~$(ii)$ and therefore
$w\in\Sch$.
\end{proof}

\section{Time-frequency representations in the
Cohen's class with kernel in $\Sch'_\omega$}
\label{secSomega}

We now want to obtain similar results in the class $\Sch_\omega$. We start by
defining the class of weights that we consider.

\begin{Def}
\label{def2}
A {\em non-quasianalytic weight function} is a continuous increasing
function $\omega:\ [0,+\infty)\to[0,+\infty)$ satisfying the
following properties:
\begin{itemize}
\item[($\alpha$)]
$\ds\exists\ L>0\ \mbox{s.t.}\ \omega(2t)\leq L(\omega(t)+1)\
  \forall t\geq0$;
\item[($\beta$)]
\vspace{0.5mm}
$\ds\int_1^{+\infty}\frac{\omega(t)}{t^2}dt<+\infty$;
\item[($\gamma$)]
\vspace{0.5mm}
$\exists\ a\in\R,\,b>0$ s.t.
$\quad\ds
\omega(t)\geq a+b\log(1+t)\qquad\forall t\geq 0$;
\item[($\delta$)]
$\varphi_\omega:\ t\mapsto\omega(e^t)$ is convex.
\end{itemize}
We then define $\omega(\xi)=\omega(\vert\xi\vert)$ for $\xi\in\C^n$.
\end{Def}

\begin{Rem}
\begin{em}
Condition $(\beta)$ is the condition of non-quasianalyticity and guarantees that the spaces $\D_{(\omega)}(K)$ defined in \eqref{domegaK} below are non-trivial for any compact set $K\subset \R^n$ with non-empty interior (see \cite[Remark 3.2(1)]{BMT}). When condition $(\beta)$ is not satisfied we say that the weight $\omega$ is
{\em quasianalytic}.
\end{em}
\end{Rem}

The function $\varphi_\omega$ of condition $(\delta)$ clearly depends on
$\omega$; for convenience we shall simply write $\varphi$ instead of
$\varphi_\omega$.

\begin{Def}
\label{def3}
  For a weight $\omega$ as in Definition \ref{def2} we define
  $\Sch_\omega(\R^n)$ as
  the set of all $u\in L^1(\R^n)$ such that $u,\widehat{u}\in C^\infty(\R^n)$ and
\begin{itemize}
\item[(i)] $\forall\lambda>0,\,\alpha\in\N^n_0:\qquad\ds
\sup_{\R^n}e^{\lambda\omega(x)}|D^\alpha u(x)|<+\infty$
\item[(ii)]
$\forall\lambda>0,\,\alpha\in\N^n_0:\qquad\ds
\sup_{\R^n}e^{\lambda\omega(\xi)}|D^\alpha\widehat{u}(\xi)|<+\infty.$
\end{itemize}
As usual, the corresponding dual space is denoted by
$\Sch^\prime_\omega(\R^n)$ and is the set of all the linear and continuous
functionals $u:\Sch_\omega(\R^n)\to\C$. We say that an element of
$\Sch_\omega^\prime(\R^n)$ is an ``$\omega$-temperate distribution''.
\end{Def}

\begin{Rem}
\label{comparison-weights}{\rm
In Definition \ref{def2} we consider weight functions in the sense of \cite{BMT}, then the weights are not necessarily subadditive in general as in \cite{B}. On the other hand, we relax condition $(\gamma)$ with respect to \cite{BMT} since we work only in the Beurling setting, as in \cite{B}.
}
\end{Rem}

Following \cite{BMT}, we define the {\em Young conjugate} $\varphi^*$ of $\varphi$ as
$$
\varphi^*(s):=\sup_{t\geq 0}\{st-\varphi(t)\},
$$
for all $s\geq 0$.
We notice that since we relax condition $(\gamma)$ with respect to
\cite{BMT}, the main properties of $\varphi^*$ hold, but
$\varphi^*(s)$ may take the value $+\infty$ for some
$s$. In this case the expressions involving $\varphi^*$ shall assume a formal
meaning; for example, if $\varphi^*(s_0)=+\infty$, then
$e^{\varphi^*(s_0)}=+\infty$, $e^{-\varphi^*(s_0)}=0$, and so on. From
Fenchel-Moreau Theorem (cf. for example \cite{BL}) we have that $\varphi^*$
is convex and $\varphi^{**}=\varphi$. Moreover, since we can assume without
loss of generality that $\omega$ vanishes on $[0,1]$ we have that
$\varphi^*(s)/s$ is increasing (cf. Lemma 1.5 of \cite{BMT}).

We state the next result, that is well-known in the case of
weights of Braun, Meise and Taylor~\cite{BMT}, and it holds also for weights
as in
Definition \ref{def2} since it is independent of condition $(\gamma)$
(for the proof we refer, for instance, to \cite[Prop. 2.1(e) and Rem. 2.2]{BJ2}):

\begin{Lemma}
\label{estimate-weight}
Let $\omega$ be a weight function
and $D$ be a constant such that
$\omega(et)\leq D\left(\omega(t)+1\right)$ for every $t\geq 0$
(such constant exists from condition $(\alpha)$).
Fix $\lambda,\rho>0$; then for every
$0<\lambda^\prime\leq \lambda/D^{[\log\rho+1]}$ we have
\beqsn
\rho^je^{\lambda\varphi^*\left(\frac j\lambda\right)}\leq
\Lambda_{\rho,\lambda} e^{\lambda'\varphi^*\left(\frac {j}{\lambda'}\right)},
\qquad\forall j\in\N_0,
\eeqsn
with $\Lambda_{\rho,\lambda}=\exp\{\lambda[\log\rho+1]\}$, where $[\log\rho+1]$
is the integer part of $\log\rho+1$.
\end{Lemma}

%\begin{proof}
%Since $\varphi^*(s)/s$ is increasing and $\lambda\geq\lambda^\prime D$ we have
%\beqs
%\label{estimate1}
%N+\lambda\varphi^*\left(\frac{N}{\lambda}\right)\leq N+\lambda^\prime
%D\varphi^*\left(\frac{N}{\lambda^\prime D}\right).
%\eeqs
%Since $\varphi(r)=\omega(e^r)$ we have
%\beqsn
%\varphi^*\left(\frac{N}{\lambda^\prime}\right)\geq &&
%\sup_{r\geq 1}\left\{ \frac{N}{\lambda^\prime}r-\varphi(r)\right\} =
%\sup_{t\geq 0}\left\{\frac{N}{\lambda^\prime}t+\frac{N}{\lambda^\prime}-
%\omega(e e^t)\right\} \\
%\geq &&\frac{N}{\lambda^\prime}-D+D\sup_{t\geq 0}
%\left\{ \frac{N}{\lambda^\prime D}t-\varphi(t)\right\} =
%\frac{N}{\lambda^\prime}-D+D\varphi^*\left(\frac{N}{\lambda^\prime D}\right).
%\eeqsn
%Combining this last estimate with \eqref{estimate1} we get the conclusion.
%\end{proof}

\begin{Rem}{\rm
\label{S-omega-S}
Observe that for $\omega_0(t)=\log(1+t)$ the corresponding space
$\Sch_{\omega_0}(\R^n)$ coincides with the classical Schwartz space
$\Sch(\R^n)$. Moreover, the condition $(\gamma)$ in Definition \ref{def2}
ensures us that for every weight $\omega$ the space $\Sch_\omega(\R^n)$ is
contained in $\Sch(\R^n)$, and so we can rewrite the definition of
$\Sch_\omega(\R^n)$ as
$$
\Sch_\omega(\R^n)=\{ u\in\Sch(\R^n)\ \text{satisfying {\rm (i)} and {\rm (ii)}
  of Definition \ref{def3}}\}.
$$}
\end{Rem}

The following
characterization of the space $\Sch_\omega$ will be useful throughout this section. The theorem below extends the carachterizations of  $\Sch_\omega$ given in \cite{CKL,GZ} and shows different equivalent systems of seminorms that can be used in such space. We begin with a lemma.

\begin{Lemma}
\label{lt}
Let $\omega$ be a weight function as in Definition~\ref{def2}. Then, for every $\lambda>0$, $k\in\N$ and $t\ge 1$ we have:
\begin{itemize}
\item[(i)]
$\ds t^k e^{-\lambda\omega(t)}\leq e^{\lambda\varphi^*\left(\frac k\lambda\right)},$

\item[(ii)]
$\ds
\inf_{j\in\N_0}t^{-j}e^{\lambda\varphi^*\left(\frac{j}{\lambda}\right)}\leq
e^{-\big(\lambda-\frac{1}{b}\big)\omega(t)-a/b},
$ where $a,b$ are the constants of condition $(\gamma)$ of Definition~\ref{def2}.
\end{itemize}
\end{Lemma}

\begin{proof}
(i) For $t\ge 1$, we have:
$$k\log t-\lambda\omega(t)\le\sup_{t\ge 1} \{k\log t-\lambda\omega(t)\}=\lambda\sup_{s\ge 0}\left\{\frac{k}{\lambda}s-\varphi(s)\right\}=\lambda\varphi^{*}\left(\frac k\lambda\right).$$
(ii) For all $s,\lambda>0$ there is $j\in\N_0$ such that
$j\leq s\lambda<j+1$ and hence (cf. \cite{BMT}):
\beqs
\nonumber
\sup_{j\in\N_0}\left\{j\log t-\lambda\varphi^*\left(\frac j\lambda\right)\right\}=&&
\lambda\sup_{j\in\N_0}\left\{\frac{j+1}{\lambda}\log t-
\varphi^*\left(\frac j\lambda\right)\right\}-\log t\\
\nonumber
\geq&&\lambda\sup_{s\ge 0}\{s\log t-\varphi^*(s)\}-\log t\\
\nonumber
=&&\lambda\varphi^{**}(\log t)-\log t
=\lambda\omega(t)-\log t\\
\label{44}
\geq&&{\Bigl(\lambda-\frac{1}{b}\Bigr)\omega(t)+\frac{a}{b}}.\nonumber
\eeqs
\end{proof}

%%%%%%%%%%%%%%%%
\begin{Th}
\label{newprop49}
Let $u\in\Sch(\R^n)$ and $\omega$ a non-quasianalytic weight function.
Then $u\in\Sch_\omega$ if and only if one of the following
equivalent conditions is satisfied:
\begin{itemize}
\item[(1)]
$u$ satisfies the conditions:
\begin{itemize}
\item[$(i)$]
$\ds\forall\lambda>0,\alpha\in\N_0^n:\ \sup_{x\in\R^n}e^{\lambda\omega(x)}
|D^\alpha u(x)|<+\infty$;
\item[$(ii)$]
$\ds\forall\lambda>0,\alpha\in\N_0^n:\ \sup_{\xi\in\R^n}e^{\lambda\omega(\xi)}
|D^\alpha \widehat{u}(\xi)|<+\infty$.
\end{itemize}
\item[(2)]
$u$ satisfies the conditions:
\begin{itemize}
\item[$(i)'$]
$\ds\forall\lambda>0,\alpha\in\N_0^n:\ \sup_{x\in\R^n}e^{\lambda\omega(x)}
|x^\alpha u(x)|<+\infty$;
\item[$(ii)'$]
$\ds\forall\lambda>0,\alpha\in\N_0^n:\ \sup_{\xi\in\R^n}e^{\lambda\omega(\xi)}
|\xi^\alpha \widehat{u}(\xi)|<+\infty$.
\end{itemize}
\item[(3)]
$u$ satisfies the conditions:
\begin{itemize}
\item[$(i)''$]
$\ds\forall\lambda>0:\ \sup_{x\in\R^n}e^{\lambda\omega(x)}
|u(x)|<+\infty$;
\item[$(ii)''$]
$\ds\forall\lambda>0:\ \sup_{\xi\in\R^n}e^{\lambda\omega(\xi)}
|\widehat{u}(\xi)|<+\infty$.
\end{itemize}
\item[(4)]
$u$ satisfies the conditions:
\begin{itemize}
\item[(a)]
$\ds \forall\beta\in\N_0^n,\lambda>0\ \exists C_{\beta,\lambda}>0:$
$$
\sup_{x\in\R^n}|x^\beta D^\alpha u(x)|
e^{-\lambda\varphi^*\left(\frac{|\alpha|}{\lambda}\right)}\leq C_{\beta,\lambda}
\qquad\forall\alpha\in\N_0^n;
$$
\item[(b)]
$\ds \forall\alpha\in\N_0^n,\mu>0\ \exists C_{\alpha,\mu}>0:$
$$
\sup_{x\in\R^n}|x^\beta D^\alpha u(x)|
e^{-\mu\varphi^*\left(\frac{|\beta|}{\mu}\right)}\leq C_{\alpha,\mu}
\qquad\forall\beta\in\N_0^n.
$$
\end{itemize}
\item[(5)]
$u$ satisfies the condition:
\begin{eqnarray*}
&&\forall\mu,\lambda>0\ \exists C_{\mu,\lambda}>0\ \mbox{s.t.}\\
&&\sup_{x\in\R^n}|x^\beta D^\alpha u(x)|
e^{-\lambda\varphi^*\left(\frac{|\alpha|}{\lambda}\right)}
e^{-\mu\varphi^*\left(\frac{|\beta|}{\mu}\right)}\leq C_{\mu,\lambda}
\qquad\forall\alpha,\beta\in\N_0^n.
\end{eqnarray*}
\item[(6)]
$u$ satisfies the condition:
\begin{eqnarray*}
&&\forall\lambda>0\ \exists C_{\lambda}>0\ \mbox{s.t.}\\
&&\sup_{x\in\R^n}|x^\beta D^\alpha u(x)|
e^{-\lambda\varphi^*\left(\frac{|\alpha+\beta|}{\lambda}\right)}
\leq C_{\lambda}
\qquad\forall\alpha,\beta\in\N_0^n.
\end{eqnarray*}
\end{itemize}
\end{Th}

\begin{proof}

  Note first that $u\in\Sch_\omega(\R^n)$ if and only if $u\in\Sch(\R^n)$ and
  satisfies $(1)$ by Remark~\ref{S-omega-S}.

  $(1)\Leftrightarrow(3)$ is Corollary 2.9 of \cite{GZ}.

  $(2)\Rightarrow(3)$ follows taking $\alpha=0$ in $(2)$.

  $(3)\Rightarrow(2)$ follows from condition $(\gamma)$ of $\omega$,
 since
  \beqsn
  |e^{\lambda\omega(x)}x^\alpha|\leq e^{-\frac{a\alpha}{b}}
  e^{\left(\frac\alpha b+\lambda\right)\omega(x)}.
  \eeqsn

$(1)\Rightarrow(4)$: let us first estimate
\beqs
\nonumber
|x^\beta D^\alpha u(x)|=&&(2\pi)^{-n}\left|\int\xi^\alpha\widehat{u}(\xi)x^\beta
e^{i\langle x,\xi\rangle}d\xi\right|\\
\nonumber
=&&(2\pi)^{-n}\left|\int\xi^\alpha\widehat{u}(\xi)D_\xi^\beta
e^{i\langle x,\xi\rangle}d\xi\right|\\
\nonumber
=&&(2\pi)^{-n}\left|\int D_\xi^\beta\left(\xi^{\alpha}\widehat{u}(\xi)\right)
e^{i\langle x,\xi\rangle}d\xi\right|\\
\nonumber
\leq&&\sum_{\afrac{\gamma\leq\beta}{\gamma\leq\alpha}}\binom\beta\gamma
\int|D_\xi^\gamma\xi^\alpha|\cdot|D_\xi^{\beta-\gamma}\widehat{u}(\xi)|d\xi\\
\nonumber
{\leq}&&\sum_{\afrac{\gamma\leq\beta}{\gamma\leq\alpha}}
\frac{\beta!}{\gamma!(\beta-\gamma)!}
\frac{\alpha!}{(\alpha-\gamma)!}\int|D_\xi^{\beta-\gamma}\widehat{u}(\xi)|
\cdot|\xi|^{|\alpha-\gamma|}d\xi\\
\label{40}
\leq&&2^{|\alpha|}\sum_{\afrac{\gamma\leq\beta}{\gamma\leq\alpha}}
\frac{\beta!}{(\beta-\gamma)!}\int
|D_\xi^{\beta-\gamma}\widehat{u}(\xi)|e^{{2\lambda}\omega(\xi)}
e^{-{\lambda}\omega(\xi)}e^{-\lambda\omega(\xi)+|\alpha-\gamma|\log|\xi|}d\xi.
\eeqs

Now, by condition $(ii)$ of $(1)$, for all $\gamma\leq \beta,$
\beqsn
|D_\xi^{\beta-\gamma}\widehat{u}(\xi)|e^{2\lambda\omega(\xi)}
\leq C_{\beta,\lambda}
\eeqsn
for some $C_{\beta,\lambda}>0$. Since we can assume without loss of generality that $|\xi|\ge 1$, we have by Lemma~\ref{lt}(i),
$$e^{-\lambda\omega(\xi)+|\alpha-\gamma|\log|\xi|}\le e^{-\lambda\omega(\xi)+|\alpha|\log|\xi|}\le e^{\lambda\varphi^*\left(\frac{|\alpha|}{\lambda}\right)}.$$ Therefore, substituting in \eqref{40}:
\beqs
\label{41}
|x^\beta D^\alpha u(x)|\leq C'_{\beta,\lambda}2^{|\alpha|}
e^{\lambda\varphi^*\left(\frac{|\alpha|}{\lambda}\right)}
\int e^{-{\lambda}\omega(\xi)}d\xi
\eeqs
for some $C'_{\beta,\lambda}>0$.

But from Lemma \ref{estimate-weight} we have that for all
$0<\lambda'\leq\lambda/D$, there exists $C_{\lambda'}>0$ such that
\beqs
\label{42}
2^{|\alpha|}e^{\lambda\varphi^*\left(\frac{|\alpha|}{\lambda}\right)}\leq
 C_{\lambda'}e^{\lambda'\varphi^*\left(\frac{|\alpha|}{\lambda'}\right)},
\eeqs
where $C_{\lambda^\prime}=e^{\lambda^\prime D}$. Moreover
\beqs
\label{43}
\int e^{-{\lambda}\omega(\xi)}d\xi\in\R {\qquad\text{for}\
  \lambda\ \text{sufficiently large},}
\eeqs
by condition $(\gamma)$.

Substituting \eqref{42} and \eqref{43} in \eqref{41} we finally have that
for all $\beta\in\N_0^n$, $\lambda'>0$ there exists $C_{\beta,\lambda'}>0$
such that
\beqsn
|x^\beta D^\alpha u(x)|\leq C_{\beta,\lambda'}
e^{\lambda'\varphi^*\left(\frac{|\alpha|}{\lambda'}\right)}
\qquad\forall\alpha\in\N_0^n,
\eeqsn
so that condition $(a)$ of $(4)$ is satisfied.

Condition $(b)$ of $(4)$ easily follows proceeding as before by condition
$(i)$ of $(1)$:
\beqsn
|x^\beta D^\alpha u(x)|e^{-\mu\varphi^*\left(\frac{|\beta|}{\mu}\right)}
\leq&&|D^\alpha u(x)|e^{|\beta|\log|x|-\mu\varphi^*\left(\frac{|\beta|}{\mu}\right)}\\
\leq&&|D^\alpha u(x)|e^{\mu\omega(x)}\leq C_{\alpha,\mu}.
\eeqsn

$(4)\Rightarrow(1)$: by $(4)$(b):
\beqsn
|D^\alpha u(x)|=&&|x^\beta D^\alpha u(x)|
e^{-\mu\varphi^*\left(\frac{|\beta|}{\mu}\right)}
e^{-|\beta|\log|x|+\mu\varphi^*\left(\frac{|\beta|}{\mu}\right)}\\
\leq&& {C_{\alpha,\mu}e^{-|\beta|\log|x|+\mu\varphi^*\left(\frac{|\beta|}{\mu}\right)}}
\qquad\forall\alpha,\beta\in\N_0^n,\ \mu>0.
\eeqsn
Now, since the constant $C_{\alpha,\mu}$ of condition $(b)$ of $(4)$ does not depend on $\beta$, by Lemma~\ref{lt}(ii) we get condition $(i)$ of
$(1)$:
\beqsn
|D^\alpha u(x)|\leq {C^\prime_{\alpha,\mu}e^{-(\mu-\frac{1}{b})\omega(x)}}
\qquad\forall\alpha\in\N_0^n,\ \mu>0,
\eeqsn
{where $C^\prime_{\alpha,\mu}=C_{\alpha,\mu}e^{-a/b}$.}
Let us now prove also condition $(ii)$ of $(1)$:
\beqs
\nonumber
|D^\beta_\xi\widehat{u}(\xi)|=&&|\widehat{x^\beta u}(\xi)|
=\left|\int x^\beta u(x)e^{-i\langle x,\xi\rangle}dx\right|\\
\nonumber
=&&\left|\int D_x^\alpha(e^{-i\langle x,\xi\rangle})\frac{1}{\xi^\alpha}
x^\beta u(x)dx\right|\\
\nonumber
=&&\left|\int\frac{1}{\xi^\alpha}D_x^\alpha\left(x^\beta u(x)\right)
e^{-i\langle x,\xi\rangle}dx\right|\\
\nonumber
\leq&&\sum_{\afrac{\gamma\leq\alpha}{\gamma\leq\beta}}\binom\alpha\gamma
\int|D_x^\gamma x^\beta|\cdot|D_x^{\alpha-\gamma}u(x)|e^{-|\alpha|\log|\xi|}dx\\
\label{45}
\leq&&\sum_{\afrac{\gamma\leq\alpha}{\gamma\leq\beta}}\binom\alpha\gamma
\frac{\beta!}{(\beta-\gamma)!}\int\langle x\rangle^{|\beta-\gamma|+n+1}
|D_x^{\alpha-\gamma}u(x)|
e^{-\lambda\varphi^*\left(\frac{|\alpha-\gamma|}{\lambda}\right)}
e^{\lambda\varphi^*\left(\frac{|\alpha-\gamma|}{\lambda}\right)-|\alpha|\log|\xi|}
\frac{1}{\langle x\rangle^{n+1}}dx
\eeqs
where $\langle x\rangle:=\sqrt{1+|x|^2}$.

By condition $(a)$ of $(4)$,
\beqs
\label{(a)}
\langle x\rangle^{|\beta-\gamma|+n+1}
|D_x^{\alpha-\gamma}u(x)|
e^{-\lambda\varphi^*\left(\frac{|\alpha-\gamma|}{\lambda}\right)}\leq
C_{\beta,\lambda}.
\eeqs
Moreover, by \eqref{42}
for all $0<\lambda'\leq\lambda/{D}$ there exists $C_{\lambda'}>0$ such
that:
\beqs
\label{(b)}
2^{|\alpha|}
e^{\lambda\varphi^*\left(\frac{|\alpha-\gamma|}{\lambda}\right)-|\alpha|\log|\xi|}
\leq C_{\lambda'}
e^{\lambda'\varphi^*\left(\frac{|\alpha|}{\lambda'}\right)-|\alpha|\log|\xi|}.
\eeqs
  Since
$\binom\alpha\gamma\leq 2^{|\alpha|}$, proceeding as before, taking the infimum in $|\alpha|$, by Lemma~\ref{lt}(ii),
we have
\beqsn
|D^\beta\widehat{u}(\xi)|\leq {C_{\beta,\lambda''}e^{-\lambda''\omega(\xi)}
\qquad\forall\beta\in\N_0^n,\,\lambda''>0}
\eeqsn
since $\int\langle x\rangle^{-n-1}dx$ is a constant.

This proves condition $(ii)$ of $(1)$.

$(5)\Rightarrow(4)$ is trivial.

$(4)\Rightarrow(5)$:
let us first remark that there are relations between the $L^\infty$
norms of $x^\beta D^\alpha u$
and the $L^2$ norms of $x^\beta D^\alpha u$. In fact, writing
$N=\left[\frac{n+1}{4}\right]+1$, we have
\beqsn
\|x^\beta D^\alpha u\|_{L^2(\R^n)}^2=&&\int|x^\beta D^\alpha u(x)|^2dx\\
=&&\int\left|x^\beta (1+|x|^2)^ND^\alpha u(x)\right|^2
\frac{1}{(1+|x|^2)^{2N}}dx\\
\leq&& c\left\|x^\beta(1+|x|^2)^ND^\alpha u\right\|_{L^\infty(\R^n)}^2
\eeqsn
for some $c>0$. We then have
\beqs
\label{L2Linfty}
\|x^\beta D^\alpha u\|_{L^2(\R^n)}^2\leq c\sum_{|\gamma|\leq N}
\frac{N!}{\gamma!(N-|\gamma|)!}\left\| x^{\beta+2\gamma}D^\alpha
u\right\|_{L^\infty(\R^n)}^2.
\eeqs

Reciprocally by Sobolev inequality (cf.
\cite[Ch. 3, Lemma 2.5]{KG}) there exists $C>0$ such that
\beqs
\label{LinftyL2}
\|x^\beta D^\alpha u\|_{L^\infty(\R^n)}\leq C\|x^\beta D^\alpha u\|_{H^s(\R^n)}
\eeqs
for $s>n/2$ (note that $x^\beta D^\alpha u\in L^\infty(\R^n)$ for every
$\alpha,\beta\in\N_0^n$ implies that $x^\beta D^\alpha u\in H^s(\R^n)$ for
every
$\alpha,\beta\in\N_0^n$ and for every $s>0$). From point (a) of (4) we
then have from \eqref{L2Linfty} that, for every $\lambda>0$,
\beqs
\label{4a-modif}
\|x^\beta D^\alpha u\|_{L^2(\R^n)}^2\leq c\sum_{|\gamma|\leq N}
\frac{N!}{\gamma!(N-|\gamma|)!} C^2_{\beta+2\gamma,\lambda}
e^{2\lambda\varphi^*\left(\frac{|\alpha|}{\lambda}\right)}=\tilde{C}^2_{\beta,\lambda}
e^{2\lambda\varphi^*\left(\frac{|\alpha|}{\lambda}\right)},
\eeqs
where
\beqsn
\tilde{C}^2_{\beta,\lambda} = c\sum_{|\gamma|\leq N}
\frac{N!}{\gamma!(N-|\gamma|)!} C^2_{\beta+2\gamma,\lambda}
\eeqsn
depends only on $\beta$, $\lambda$ and the dimension $n$.
Now, from \eqref{L2Linfty}, the point (b) of (4)
(rewritten for convenience with $\mu^\prime$ instead of $\mu$) implies that
\beqsn
\|x^\beta D^\alpha u\|_{L^2(\R^n)}^2\leq c\sum_{|\gamma|\leq N}
\frac{N!}{\gamma!(N-|\gamma|)!} C^2_{\alpha,\mu^\prime}
e^{2\mu^\prime\varphi^*\left(\frac{|\beta+2\gamma|}{\mu^\prime}\right)};
\eeqsn
from the convexity of $\varphi^{*}$ we get:
\beqsn
e^{\mu^\prime\varphi^*\left(\frac{|\beta+2\gamma|}{\mu^\prime}\right)}\leq
e^{\frac{\mu^\prime}{2}\varphi^*\left(\frac{2|\beta|}{\mu^\prime}\right)}
e^{\frac{\mu^\prime}{2}\varphi^*\left(\frac{4|\gamma|}{\mu^\prime}\right)}.
\eeqsn
Then we obtain
\beqsn
\|x^\beta D^\alpha u\|_{L^2(\R^n)}^2\leq c\sum_{|\gamma|\leq N}
\frac{N!}{\gamma!(N-|\gamma|)!} C^2_{\alpha,\mu^\prime}
e^{\mu^\prime\varphi^*\left(\frac{2|\beta|}{\mu^\prime}\right)}
e^{\mu^\prime\varphi^*\left(\frac{4|\gamma|}{\mu^\prime}\right)} = C^\prime_{\alpha,\mu^\prime}
e^{\mu^\prime\varphi^*\left(\frac{2|\beta|}{\mu^\prime}\right)},
\eeqsn
where
\beqsn
C^\prime_{\alpha,\mu^\prime}=c\sum_{|\gamma|\leq N}\frac{N!}{\gamma!(N-|\gamma|)!}
C^2_{\alpha,\mu^\prime}
e^{\mu^\prime\varphi^*\left(\frac{4|\gamma|}{\mu^\prime}\right)}
\eeqsn
depends only on $\alpha$, $\mu^\prime$ and the dimension $n$.
Then, writing $\mu:=\mu^\prime/2$ we obtain that for every $\alpha\in \N_0^n$
and for every $\mu>0$ there exists a constant $\tilde{C}_{\alpha,\mu}>0$
satisfying
\beqs
\label{4b-modif}
\|x^\beta D^\alpha u\|_{L^2(\R^n)}\leq \tilde{C}_{\alpha,\mu}
e^{\mu\varphi^*\left(\frac{|\beta|}{\mu}\right)}.
\eeqs
Now, we will use that
\begin{equation}
\gamma!\leq C_{\lambda} e^{\lambda \varphi^{*}(\frac{|\gamma|}{\lambda})},\label{fact}
\end{equation}
for all $\lambda>0$, $\gamma\in\N_{0}^{n}$ and some constant $C_{\lambda}$. This is true because $\omega(t)=o(t)$ as $t\to\infty$ (from condition $(\beta)$ of Definition~\ref{def2}). Therefore, from \eqref{4a-modif} and \eqref{4b-modif}, and following the same
idea as in \cite{CCK}, we thus estimate:
\beqsn
\|x^\beta D^\alpha u\|_{L^2(\R^n)}^2=&&\int_{\R^n}\left|\left(
x^{2\beta}\partial_x^\alpha u(x)
\right)\cdot\partial_x^\alpha u(x)\right|dx\\
\leq &&\sum_{\afrac{\gamma\leq2\beta}{\gamma\leq\alpha}}\binom\alpha\gamma
\binom{2\beta}{\gamma}\gamma!\|\partial^{2\alpha-\gamma}u(x)\|_{L^2(\R^n)}
\|x^{2\beta-\gamma}u(x)\|_{L^2(\R^n)}\\
\leq&&\sum_{\afrac{\gamma\leq2\beta}{\gamma\leq\alpha}}\binom\alpha\gamma
\binom{2\beta}{\gamma}\gamma!
\tilde{C}_{0,\lambda}e^{\lambda\varphi^*\left(\frac{|2\alpha-\gamma|}{\lambda}\right)}
\tilde{C}_{0,\mu}e^{\mu\varphi^*\left(\frac{|2\beta-\gamma|}{\mu}\right)}\\
\leq&&\sum_{\afrac{\gamma\leq2\beta}{\gamma\leq\alpha}}\binom\alpha\gamma
\binom{2\beta}{\gamma} C_\lambda e^{\lambda\varphi^*\left(\frac{|\gamma|}{\lambda}\right)}
e^{\lambda\varphi^*\left(\frac{|2\alpha-\gamma|}{\lambda}\right)}
\tilde{C}_{0,\mu}e^{\mu\varphi^*\left(\frac{|2\beta|}{\mu}\right)}\\
\leq&&2^{|\alpha|}2^{2|\beta|}C_\lambda \tilde{C}_{0,\mu}
e^{\lambda\varphi^*\left(\frac{|2\alpha|}{\lambda}\right)}
e^{\mu\varphi^*\left(\frac{|2\beta|}{\mu}\right)}\\
\leq&&C_{\lambda',\mu'}e^{\lambda'\varphi^*\left(\frac{|2\alpha|}{\lambda'}\right)}
e^{\mu'\varphi^*\left(\frac{|2\beta|}{\mu'}\right)}
\eeqsn
for some $\tilde{C}_{0,\lambda},\tilde{C}_{0,\mu},C_\lambda,C_{\lambda',\mu'}>0$,
because of the properties of $\varphi^*$ and \eqref{42}.
Extracting the square root and writing $\lambda=\lambda^\prime/2$ and
$\mu=\mu^\prime/2$ we have that for every $\lambda,\mu>0$ there exists a
constant $\tilde{C}_{\lambda,\mu}>0$ such that
\beqs
\label{5-modif}
\|x^\beta D^\alpha u\|_{L^2(\R^n)}\leq \tilde{C}_{\lambda,\mu}
e^{\lambda\varphi^*\left(\frac{|\alpha|}{\lambda}\right)}
e^{\mu\varphi^*\left(\frac{|\beta|}{\mu}\right)}.
\eeqs
In order to prove that (5) holds, we have to estimate
$\|x^\beta D^\alpha u\|_{L^\infty(\R^n)}$. Fix $\bar{s}=\left[\frac{n}{2}\right]+1$;
from \eqref{LinftyL2} and \eqref{5-modif} we have
\beqs
\|x^\beta D^\alpha u\|_{L^\infty(\R^n)}\leq&&C \sum_{|\gamma|\leq \bar{s}}
\|D^\gamma (x^\beta D^\alpha u)\|_{L^2(\R^n)} \nonumber \\
\leq &&C\sum_{|\gamma|\leq \bar{s}}\sum_{\afrac{\sigma\leq\gamma}{\sigma\leq\beta}}
\binom{\gamma}{\sigma} \binom{\beta}{\sigma} \sigma! \| x^{\beta-\sigma}
D^{\alpha+\gamma-\sigma}u\|_{L^2(\R^n)} \label{5-Linfty} \\
\leq &&C\sum_{|\gamma|\leq \bar{s}}\sum_{\afrac{\sigma\leq\gamma}{\sigma\leq\beta}}
\binom{\gamma}{\sigma} \binom{\beta}{\sigma} \sigma! \tilde{C}_{\lambda,\mu}
e^{\lambda\varphi^*\left(\frac{|\alpha+\gamma-\sigma|}{\lambda}\right)}
e^{\mu\varphi^*\left(\frac{|\beta-\sigma|}{\mu}\right)}. \nonumber
\eeqs
Now, proceeding as in previous steps, using inequality (\ref{fact}),
the convexity of $\varphi^{*}$ and similar properties as before we
easily get (5).

$(5)\Leftrightarrow(6)$ is trivial from the convexity of $\varphi^{*}$.
\end{proof}

%%%%%%%%%%%%%%%%

We recall quickly the definition of
the space $\E_{(\omega)}(\Omega)$ of
$\omega$-ultradifferentiable functions of Beurling type in an
open subset $\Omega$ of $\R^n$. It is the set
\beqsn
\E_{(\omega)}(\Omega):=\Big\{f\in C^\infty(\Omega):\ &&\forall K\subset\subset
\Omega,\ \forall m\in\N\\
&&\sup_{\alpha\in\N^n}\sup_{x\in K}|D^\alpha f(x)|
e^{-m\varphi^*\left(\frac{|\alpha|}{m}\right)}<+\infty\Big\}.
\eeqsn
To define then the space of
$\omega$-ultradifferentiable functions of Beurling type with compact support,
we first consider, for a compact set $K\subset\Omega$,
\begin{equation}\label{domegaK}
\D_{(\omega)}(K):=\{f\in\E_{(\omega)}(\Omega):\ \supp f\subseteq K\}.
\end{equation}
This space is not trivial because of $(\beta)$ of Definition~\ref{def2} (considering the non-quasianalytic case; for quasianalytic weights the space \eqref{domegaK} contains only the function identically $0$). Finally, we set the space of test functions as follows
\beqsn
\D_{(\omega)}(\Omega) = \indlim_{K\nearrow\Omega} \D_{(\omega)}(K).
\eeqsn

The spaces of
Roumieu type are not used here and a definition can be found in \cite{BMT} with a stronger condition instead of our $(\gamma)$. The use of $(\gamma)$ is clarified for the Beurling case in \cite{BG} (see also \cite{F}).

We recall here some properties of the space $\Sch_\omega(\R^n)$, that we shall
use in the following. For the proofs we refer to \cite[Kap. I, \S6]{F} (see
also \cite{B})
.
\begin{Prop}
\label{PropertiesSomega}
Let $\omega$ be as in Definition \ref{def2}.
\begin{enumerate}[\indent \rm (a)]
\item
  The Fourier transform is a continuous automorphism
$\F :\Sch_\omega(\R^n)\to  \Sch_\omega(\R^n)$.
It can be extended to $\Sch_\omega^\prime(\R^n)$ in the
  standard way, by the formula
\beqsn
\langle\widehat{u},\varphi\rangle=\langle u,\widehat{\varphi}\rangle
\qquad\forall\varphi\in\Sch_\omega.
\eeqsn
\item
$\Sch_\omega(\R^n)$ is an algebra under multiplication and convolution.
\item
The differentiation $D^\alpha$, the multiplication by $x^\alpha$,
for $\alpha\in\N^n_0$, the multiplication by $e^{i\langle \cdot,a\rangle}$
and the translation $\tau_a$ acting as $\tau_a u(x) := u(x-a)$, for
$a\in\R^n$, are continuous on $\Sch_\omega(\R^n)$.
\item
The following inclusions hold: $\D_{(\omega)}(\R^n)\subset\Sch_\omega(\R^n)
\subset\E_{(\omega)}(\R^n)$.
\item
$\D_{(\omega)}(\R^n)$ is dense in $\Sch_\omega(\R^n)$.
\item
For $\psi\in\Sch_\omega(\R^n)$ and $u\in\Sch^\prime_\omega(\R^n)$ we have
$\psi*u\in\Sch^\prime_\omega(\R^n)$ and
$\widehat{\psi*u}=\widehat\psi\cdot\widehat u$.
\end{enumerate}
\end{Prop}

We observe that Theorem~\ref{newprop49} allows to define equivalent systems of seminorms for $\mathcal{S}_{\omega}$. For example, from condition (6) of this theorem it is clear that, given $u\in\mathcal{S}_{\omega}$ the family
$$
p_{\lambda}(u):=\sup_{\alpha,\beta\in\N_{0}^{n}}\sup_{x\in\R^n}|x^\beta D^\alpha u(x)|
e^{-\lambda\varphi^*\left(\frac{|\alpha+\beta|}{\lambda}\right)},
$$
for all $\lambda>0$, defines a fundamental system of seminorms for $\mathcal{S}_{\omega}.$ In a similar way, we can construct different equivalent systems of seminorms from the other conditions of the theorem.

\begin{Rem}
\label{reminvertibilita}
  \begin{em}
  By Proposition~\ref{PropertiesSomega}~(a),
    $\Sch_\omega(\R^{2n})$ is invariant by Fourier transform $\F=\F_{(x,y)}$.

    Moreover, it can be proved by direct calculation that $\Sch_\omega(\R^{2n})$
    is also invariant by partial Fourier transform $\F_x$. This can also be deduced
    from the facts that it is clear for $\varphi\in\Sch_\omega(\R^{2n})$ of
the form
    $\varphi(x,y)=f(x)\cdot g(y)$, with $f,g\in\Sch_\omega (\R^n)$, and $\Sch_\omega(\R^n)\otimes\Sch_\omega(\R^n)$ is dense in
  $\Sch_\omega(\R^{2n})$ by
    Proposition~\ref{PropertiesSomega}~(e) and \cite[Thm. 8.1]{BMT}
    (cf. also \cite{BG}, since we assume condition $(\gamma)$ of Definition~\ref{def2}
    instead of $\log(t)=o(\omega(t))$ as $t\to\infty$).

    Furthermore, the linear change of variable $T:\ \Sch_\omega\to\Sch_\omega$
defined in \eqref{wigner-comp-op} is invertible and therefore
from \eqref{wigner-fourier} we deduce that also the Wigner transform
\beqsn
\Wig:\ &&\Sch_\omega\longrightarrow\Sch_\omega\\
&&\Sch'_\omega\longrightarrow\Sch'_\omega
\eeqsn
is invertible.
      \end{em}
\end{Rem}
The following lemma can be deduced as Lemma~\ref{lemma4} above.

\begin{Lemma}
\label{cor3}
If $\varphi\in C^\infty_p(\R^n)$ and $u\in\Sch_\omega$
then $\varphi u\in\Sch_\omega$. If $w\in\Sch'_\omega$
then $\varphi w\in\Sch'_\omega$.
\end{Lemma}
%
%\begin{proof}
%We use Theorem~\ref{newprop49}. Assuming that $u$ satisfies
%conditions $(a)$ and $(b)$ of $(4)$ we prove that also $\varphi u$ satisfies
%the same two conditions, arguing as in the proof of Lemma~\ref{lemma4}.
%The statement for $w\in\Sch'$ follows from the one with $u\in\Sch$, as in the
%proof of Lemma~\ref{lemma4}.
%\end{proof}

\begin{Prop}
\label{prop412}
For every non-quasianalytic weight function
$\omega$ we have  $$\Sch^\prime(\R^n)\subset\Sch_\omega^\prime(\R^n).$$
\end{Prop}

\begin{proof}
  We already know that $\Sch_\omega(\R^n)\subset\Sch(\R^n)$,
  cf. Remark \ref{S-omega-S}. It is enough to prove that $\Sch_\omega(\R^n)$
  is dense in $\Sch(\R^n)$.

By \cite[Prop. 3.9]{BMT} we have that $\D_{(\omega)}(\R^n)$
is dense in $\D(\R^n)$.
On the other hand, it is known that $\D(\R^n)$ is dense in the Schwartz
class $\Sch(\R^n)$.
Then $\D_{(\omega)}(\R^n)$ is also dense in $\Sch(\R^n)$.
From the inclusions
\beqsn
\D_{(\omega)}(\R^n)\hookrightarrow
\Sch_\omega(\R^n)\hookrightarrow\Sch(\R^n),
\eeqsn
we can conclude that $\Sch_\omega$ is dense in $\Sch$.
\end{proof}

We give now the definition of regularity in the $\Sch_\omega$ frame and we
extend to $\Sch_\omega$ the results of Sections \ref{sec2} and \ref{sec3}.

\begin{Def}
\label{defomegaregular}
A linear operator $A$ on $\Sch'_\omega(\R^n)$ is {\em $\omega$-regular} if
\beqsn
Au\in\Sch_\omega(\R^n)\quad \Rightarrow\quad
u\in\Sch_\omega(\R^n),\qquad\forall u\in\Sch'_\omega(\R^n).
\eeqsn

\end{Def}

\begin{Prop}
\label{prop2}
Let $\sigma=q(D_1,D_2)\F^{-1}(e^{-iP(\xi,\eta)})$ for some
$P(\xi,\eta)\in\R[\xi,\eta]$  and $q(\xi,\eta)\in\C[\xi,\eta]$ with
$q(\xi,\eta)\neq0$ for all $\xi,\eta\in\R$. Let $u\in\Sch'_\omega$ for a
non-quasianalytic weight function $\omega$. Then $Q[u]=\sigma*\Wig[u]$ is
well defined and satisfies:
\begin{itemize}
\item[(i)]
$Q:\ \Sch'_\omega\to\Sch'_\omega$ is invertible;
\item[(ii)]
$Q$ is $\omega$-regular;
\item[(iii)]
$Q:\ \Sch_\omega\to\Sch_\omega$.
\end{itemize}
\end{Prop}

\begin{proof}
The proof is analogous to that of Lemma~\ref{lemma2}
(or Proposition~\ref{prop3}), because of the invertibility of the Fourier
transform and of the Wigner transform on $\Sch_\omega$ and $\Sch'_\omega$
(cf. Remark~\ref{reminvertibilita}), and by
means of Lemma~\ref{cor3}, since $\widehat{\sigma},
1/\widehat{\sigma}\in C^\infty_p$ and
$|\widehat{\sigma}(\xi,\eta)|=|q(\xi,\eta)|\neq0$ for
all $\xi,\eta\in\R$.
\end{proof}

\begin{Th}
\label{th5}
Let $\omega$ be a non-quasianalytic weight function,
$P(\xi,\eta)\in\R[\xi,\eta]$ and $q(\xi,\eta)\in\C[\xi,\eta]$ with
$q(\xi,\eta)\neq0$ for all $\xi,\eta\in\R$.
Let $\sigma=\F^{-1}(e^{-iP(\xi,\eta)})\in\Sch'\subset\Sch'_\omega$,
$\sigma_1=q(D_1,D_2)\sigma$,
$Q^{(\sigma)}[w]=\sigma*\Wig[w]$ and $Q^{(\sigma_1)}[w]=\sigma_1*\Wig[w]$ for
$w\in\Sch'_\omega$.
Then, if $B(x,y,D_x,D_y)$ is a linear partial differential operator
with polynomial coefficients, we have that
\beqs
\label{48}
Q^{(\sigma_1)}[Bw]=\widetilde{AB}Q^{(\sigma)}[w],
\eeqs
where $A$ is the operator defined by $A(M_1,M_2,D_1,D_1)=q(D_1+D_2,M_2-M_1)$
and $\widetilde{AB}$ is obtained from $AB$ as in \eqref{31}.
Moreover, $B$ is $\omega$-regular if and only if
$\widetilde{AB}$ is $\omega$-regular.
\end{Th}

\begin{proof}
Formula \eqref{48} has already been proved in Theorem~\ref{cor2}.

Let $B$ be $\omega$-regular and prove that $\widetilde{AB}$ is
$\omega$-regular. So take $u\in\Sch'_\omega$ and assume that
$\widetilde{AB}u\in\Sch_\omega$. By Proposition~\ref{prop2}~$(i)$ (with
$q(\xi,\eta)\equiv1$) there
exists $w\in\Sch'_\omega$ such that $u=Q^{(\sigma)}[w]$.
By \eqref{48} we have that
$Q^{(\sigma_1)}[Bw]=\widetilde{AB}Q^{(\sigma)}[w]
=\widetilde{AB}u\in\Sch_\omega$ and hence
$Bw\in\Sch_\omega$ by Proposition~\ref{prop2}~$(ii)$.
Since $B$ is $\omega$-regular by assumption, $w\in\Sch_\omega$.
Finally, from Proposition~\ref{prop2}~$(iii)$, we have that
$u=Q^{(\sigma)}[w]\in\Sch_\omega$ and we have
proved that $\widetilde{AB}$ is $\omega$-regular.

Reciprocally, assuming that $\widetilde{AB}$ is $\omega$-regular, if
$Bu\in\Sch_\omega$ for some $u\in\Sch'_\omega$, then
$Q^{(\sigma_1)}[Bu]\in\Sch_\omega$
by Proposition~\ref{prop2}~$(iii)$ and therefore
$\widetilde{AB}Q^{(\sigma)}[u]=Q^{(\sigma_1)}[Bu]\in\Sch_\omega$.
By the $\omega$-regularity of $\widetilde{AB}$ we have that
$Q^{(\sigma)}[u]\in\Sch_\omega$ and hence $u\in\Sch_\omega$ by
Proposition~\ref{prop2}~$(ii)$. This proves that $B$ is
$\omega$-regular.
\end{proof}

\begin{Rem}
\label{rem417}
\begin{em}
  Theorem \ref{th5} is an extension to $\mathcal{S}_\omega$ of Theorem \ref{cor2}. Observe in particular that for $q\equiv1$, and hence $A\equiv I$, Theorem~\ref{th5} implies that $B$
  is $\omega$-regular if and only if $\tilde{B}$ is $\omega$-regular, extending
  therefore to $\Sch_\omega$, for every weight function $\omega$, the
  results obtained for $\Sch$ in the previous sections.
  \end{em}
  \end{Rem}

\begin{Rem}
\begin{em}
All the results of the present section may be proved also in the quasianalytic case,
and more precisely when the weight function $\omega$ satisfies
$\omega(t)=o(t)$ as $t\to+\infty$,
instead of $(\beta)$.
In this case $\Sch_\omega$ does not contain functions with compact support,
so that conditions $(d)$ and $(e)$ of Proposition~\ref{PropertiesSomega} will drop.
However, Proposition~\ref{prop412} is still valid, since the density of $\Sch_\omega(\R^n)$
in $\Sch(\R^n)$ can be proved by \cite[Lemma 3.2]{L}, which shows that the
Hermite functions, that are a Shauder basis in $\Sch(\R^n)$, are in
$\Sch_\omega(\R^n)$ because of Theorem~\ref{newprop49}(6) and the following
property:
\beqsn
\forall B>0,\,\lambda>0\ \exists C_{B,\lambda}>0\ \mbox{s.t.}\qquad
B^nn!\leq C_{B,\lambda}e^{\lambda\varphi^*\left(\frac{n}{\lambda}\right)},
\quad\forall n\in\N,
\eeqsn
which follows from \eqref{fact} and Lemma~\ref{estimate-weight}.
\end{em}
\end{Rem}

\section{Examples}

In this section we give some examples of applications of our results in
order to find classes of regular partial differential operators with
polynomial coefficients. Recall from \cite{S} that a polynomial $a(x,\xi)$
of order $m$, with $x,\xi\in\R^n$, is said to be hypoelliptic if there exists
$m^\prime\leq m$, $\rho\in (0,1]$, and positive constants $c,C$ such that for
every $\alpha,\beta\in\N_0$,
\begin{equation}
\begin{split}
\label{g-hypoelliptic}
\vert a(x,\xi)\vert &\geq c\langle (x,\xi)\rangle^{m^\prime} \\
\vert\partial^\alpha_x \partial^\beta_\xi a(x,\xi)\vert &\leq C\vert
a(x,\xi)\vert \,\langle(x,\xi)\rangle^{-\rho(\vert\alpha\vert+\vert\beta\vert)}
\end{split}
\end{equation}
for $| (x,\xi)|\geq B$, where $\langle (x,\xi)\rangle
:=(1+\vert x\vert^2+\vert\xi\vert^2)^{1/2}$.

\begin{Rem}
\begin{em}
From the results of \cite{S}
we have that an operator with polynomial coefficients $a(x,D)$ whose symbol
$a(x,\xi)$ is hypoelliptic, is regular in $\Sch(\R^n)$, in the sense that
it satisfies the condition of Definition \ref{def-reg-S}. The question of
proving regularity for non-hypoelliptic operators is not trivial.
The results of the previous sections enable to find classes of regular
(but not hypoelliptic) operators, and these classes are quite large due to
the freedom we have in choosing the kernel $\sigma$ of the representation in
the Cohen's class. For example, using Theorem \ref{th4}, we could consider a
regular (possibly hypoelliptic) operator $B$ and we immediately have
regularity of the corresponding $\tilde{B}$, cf. \eqref{31}. The operator
$\tilde{B}$ in general is not hypoelliptic (cf. Remark \ref{remex54} or
\cite{BO} for more general examples of hypoelliptic operators $B$
that are transformed, in the simple case when $\sigma$ is the Dirac
distribution, into regular operators $\tilde{B}$ which are never hypoelliptic).
\end{em}
\end{Rem}

It will be useful, for the discussion of examples, the following
\begin{Prop}
  \label{propBxI}
  Let $\omega$ be a non-quasianalytic weight function and let $B$ be
  a continuous linear operator on $\Sch_\omega'(\R)$ such that
  $B(\Sch_\omega(\R))\subseteq\Sch_\omega(\R)$. Let $I$ be the indentity
  operator on
  $\Sch'(\R)$ and consider the operator $B\hat{\otimes}I$, interpreted as the
  ``extension of B from one variable in $\R$ to two variables in $\R^2$''.
  If $B\hat{\otimes}I$ is $\omega$-regular in $\Sch_\omega'(\R^2)$, then $B$
  is $\omega$-regular and injective in $\Sch_\omega'(\R)$.
  \end{Prop}

{\em Proof:}
Let $u\in\Sch_\omega'(\R)$ with $Bu\in\Sch_\omega(\R)$. We prove that
$u\in\Sch_\omega(\R)$. Indeed, for all $v\in\Sch_\omega(\R)$ we have that
$(B\hat{\otimes}I)(u\otimes v)=(Bu)\otimes v\in\Sch_\omega(\R^2)$, since
$Bu\in\Sch_\omega(\R)$. Then $u\otimes v\in\Sch_\omega(\R^2)$ for every $v\in\Sch_\omega(\R)$, because
$B\hat{\otimes}I$ is regular by assumption, and hence $u\in\Sch_\omega(\R)$.
This proves that $B$ is $\omega$-regular.

To prove that $B$ is injective let us assume by contradiction that there
exists $u\in\Sch_\omega'$ with $u\neq0$ such that $Bu=0$.
Then, for the Dirac distribution $\delta$ we have that
$(B\hat{\otimes}I)(u\otimes\delta)=0\in\Sch_\omega$ but
$u\otimes\delta\notin\Sch_\omega$, and hence $B\hat{\otimes}I$ would not
be regular.
\hfill{\large $\Box$}
\vskip\baselineskip
  Proposition \ref{propBxI} has already been proved in \cite{BO} in the
  Schwartz
  case, i.e. $\omega(t)=\log(1+t)$. Under suitable assumptions also the
  converse is true in $\Sch'$, as it was proved in \cite[Thm. 3]{BO}.

\begin{Ex}{\rm
As first example consider the simple cases of a multiplication operator
$$
B(x,y,D_x,D_y) = b(x,y),
$$
where $b$ is a polynomial. Then it is easy to prove that $B$ is regular
if and only if $b$ never vanishes. We then have from Theorems \ref{th4} and
\ref{th5} (cf. also Remark \ref{rem417}) that the operator
$$
\tilde{B}=b\left(M_1-\frac{1}{2}D_2-P_1,M_1+\frac{1}{2}D_2-P_1\right)
$$
is $\omega$ regular for each weight $\omega$, for every $P_1$ as in
\eqref{38}; in particular it is regular in the sense of Schwartz spaces.
Observe that $P_1=P_1(D_1,D_2)$ is in fact an arbitrary partial differential
operator with real constant coefficients in two variables. In the special
case when the polynomial $b$ depends only on one variable, we get that, if
$b$ never vanishes, the operator
\beqs
\label{ex1}
b\left( x+P(D_x,D_y)\right)
\eeqs
is regular in Schwartz spaces and $\omega$ regular, for every partial
differential operator $P=P(D_x,D_y)$ with constant real coefficients, without
any other assumption on $P$.}
\end{Ex}

The twisted Laplacian is an important example of a non hypoelliptic but
regular operator. Its regularity (in Schwartz spaces) was proved in \cite{W}
and then re-obtained in \cite{BO} as a particular case of operators obtained
as Wigner transformation of the harmonic oscillator. Applying the
transformations in the Cohen's class considered in this paper we have the
following example.
\begin{Ex}{\rm
\label{ex-2}
We know from \cite{BO} that the operator
\beqs
\label{Harm-Os}
B(x,y,D_x,D_y)=x^2+D_x^2
\eeqs
is $\Sch$-regular, since it is the tensor product $B_1\hat{\otimes}I$
of the harmonic oscillator $B_1(x,D_x)=x^2+D_x^2$ in the $x$-variable
(that is regular and one-to-one) and the identity in the $y$-variable.
Then from Theorem \ref{th4} we have that the operator
\beqs
\label{H-O-1}
\tilde{B}=\left(M_1-\frac{1}{2}D_2-P_1(D_1,D_2)\right)^2+\left(M_2+\frac{1}{2}D_1-P_2(D_1,D_2)\right)^2
\eeqs
is regular in Schwartz spaces, where
\beqsn
P_1=(iD_1P)(D_1,D_2),\qquad P_2=(iD_2P)(D_1,D_2)
\eeqsn
and $P$ is an arbitrary polynomial with real coefficients. In particular, if $P$ is of the form $P(\xi,\eta)=P^{(1)}(\xi)+P^{(2)}(\eta)$, then $P_1$ and $P_2$ are arbitrary operators in $D_1$ and $D_2$, respectively, and so we have that the
operator
\beqs
\label{H-O-2}
\left(x-\frac{1}{2}D_y+Q(D_x)\right)^2+\left(y+\frac{1}{2}D_x+R(D_y)\right)^2
\eeqs
is regular in Schwartz spaces, for arbitrary partial differential operators
$Q(D_x)$ and $R(D_y)$ with constant real coefficients. Another particular case
of \eqref{H-O-1} is when
$P(\xi,\eta)=\frac{1}{2}\xi\eta+P^{(1)}(\xi)+P^{(2)}(\eta)$ for polynomials
$P^{(1)}$ and $P^{(2)}$ with real coefficients, and in this case we get the
$\Sch$-regularity of
\beqs
\label{H-O-3}
\left(x-D_y+Q(D_x)\right)^2+\left(y+R(D_y)\right)^2,
\eeqs
for arbitrary differential operators $Q(D_x)$ and $R(D_y)$ with constant real
coefficients.

The same results hold in the $\Sch_\omega$ frame,
for a non-quasianalytic weight function $\omega$. In order to prove this, we
can show that, using the same technique as in \cite{W}, the
twisted Laplacian
\beqs
\label{twisted}
L=\left( D_x-\frac{1}{2}y\right)^2+\left( D_y+\frac{1}{2}x\right)^2
\eeqs
is $\omega$-regular for every weight $\omega$.

To this aim we first prove, following \cite[Prop. 6.2]{W}, that
there exists a constant $c>0$ and, for every $s>0$, there exists $C_s>0$
such that
 \beqs
 \label{defg}
 g(w):=\frac{1}{4\pi}\int_0^{+\infty}e^{-\frac14|w|^2\cosh t}dt
 \leq C_s \frac{1}{|w|^s}e^{-c|w|^2}\qquad\forall w\in\R^2\setminus\{0\}.
 \eeqs
 Indeed, for all $w\in\R^2\setminus\{0\}$,
 \beqs
 \label{defg2}
 g(w)\leq\frac{1}{4\pi}\int_0^{+\infty}e^{-\frac14|w|^2\frac{e^t}{2}}dt
 =\frac{1}{4\pi}\int_{|w|^2/8}^{+\infty}\frac{e^{-y}}{y}dy.
 \eeqs
 We then have that, for $0<\vert w\vert\leq 1$,
 \beqsn
 g(w) &&\leq \frac{1}{4\pi}\left(\int_{|w|^2/8}^1 \frac{e^{-y}}{y}dy +
\int_1^{+\infty}\frac{e^{-y}}{y}dy\right)\leq \frac{e^{-\vert w\vert^2/8}}{4\pi}
\int_{|w|^2/8}^1 \frac{1}{y}dy+D^\prime \\
 &&=-\frac{1}{4\pi}\log\frac{\vert w\vert^2}{8}e^{-\vert w\vert^2/8}+D^\prime \leq
D^{\prime\prime}\left( 1-\log\frac{\vert w\vert^2}{8}\right)e^{-\vert w\vert^2/8};
 \eeqsn
 then for every $s>0$ we can find a positive constant $C^\prime_s$ such that
for every $0<\vert w\vert\leq 1$
 \beqs
 \label{defg1}
 g(w)\leq C^\prime_s\frac{1}{\vert w\vert^s}e^{-\vert w\vert^2/8}.
 \eeqs
 Consider now $w\in\R^2$ such that $\vert w\vert\geq 1$. From \eqref{defg2}
we get
 \beqsn
 g(w)\leq \frac{2}{\pi\vert w\vert^2}\int_{|w|^2/8}^{+\infty} e^{-y}dy =
\frac{2}{\pi\vert w\vert^2}e^{-\vert w\vert^2/8}.
 \eeqsn
 Then, if we fix $c<1/8$, for every $s>0$ we can find a positive constant
$C^{\prime\prime}_s$ such that
 \beqs
 \label{defg3}
 g(w)\leq C^{\prime\prime}_s \frac{1}{\vert w\vert^s}e^{-c\vert w\vert^2}
 \eeqs
 for all $\vert w\vert\geq 1$. From \eqref{defg1} and \eqref{defg3} we finally
have that \eqref{defg} is satisfied for every $w\in\R^2\setminus\{0\}$, with
$c$ as in \eqref{defg3} and $C_s=\max\{C^\prime_s,C^{\prime\prime}_s\}$. \\
 We prove now, following \cite[Thm. 6.1]{W}, that if $f\in\Sch_\omega(\R^2)$
 then the solution $u$ of $Lu=f$ satisfies $(i)$ and $(ii)$ of
 Definition~\ref{def3}.
 Indeed, from \cite{W} we have that
 \beqsn
 u(z)=\int_{\R^2} g(w)e^{\frac12 i(z_2w_1-z_1w_2)}f(z-w)dw\qquad\forall z\in\R^2,
 \eeqsn
 where $g$ is defined in \eqref{defg}. By condition $(\alpha)$ in Definition~\ref{def2}, there is some constant $K>1$ (see \cite[1.2 Lemma]{BMT}) such that, for $\beta\in\mathbb{Z}^2_+$ we have
 \beqs
 \label{I1}
 |e^{\lambda\omega(z)}(\partial_z^\beta u)(z)|\leq \int_{\R^2}
e^{\lambda K(\omega(w)+1)}\vert g(w)\vert e^{\lambda K\omega(z-w)}\left\vert
\partial^\beta_z\left(e^{\frac{i}{2}(z_2w_1-z_1w_2)}f(z-w)\right)\right\vert\,dw.
 \eeqs
 The latter integral can be estimated by a sum of terms of the kind
 \beqsn
 \int_{\R^2} e^{\lambda K(\omega(w)+1)}\vert w\vert^{\vert\alpha\vert}\vert g(w)
 \vert e^{\lambda K\omega(z-w)}\left\vert \partial^\gamma_{z}
f(z-w)\right\vert\,dw,
 \eeqsn
 with $\alpha,\gamma\leq\beta$.
   Note that $e^{\lambda K\omega(z-w)}|\partial^\gamma_z f(z-w)|$ is bounded because
   $f=Lu\in\Sch_\omega(\R^2)$, moreover \eqref{defg} implies that
   $e^{\lambda K(\omega(w)+1)}\vert w\vert^{\vert\alpha\vert}\vert g(w)\vert$
   is summable either
   in $\{w\in\R^2:\ |w|\leq1\}$ for $s<2$, or in
   $\{w\in\R^2:\ |w|\geq1\}$ since $\omega(t)=o(t)$ by  condition $(\beta)$. Therefore
   $\sup_z |e^{\lambda\omega(z)}(\partial_z^\beta u)(z)|<+\infty$,
   and so $u$ satisfies $(i)$ of Definition~\ref{def3}.

In order to prove that $u$
satisfies also (ii) of Definition \ref{def3} we observe that $u$ satisfies
$Lu=f$ if and only if $\hat{u}$ satisfies $\hat{L}\hat{u}=\hat{f}$, where
$$
\hat{L}=\left( \frac{1}{2}D_\eta+\xi\right)^2+
\left( \frac{1}{2}D_\xi-\eta\right)^2,
$$
and this happens if and only if $v(\xi,\eta):=\hat{u}(\xi/2,\eta/2)$
satisfies the equation $Lv(2\xi,2\eta)=\hat{f}(\xi,\eta)$. Since the
dilations do not affect the estimates (i) and (ii) of Definition \ref{def3}
due to the fact that $\lambda$ is arbitrary, we then have from the previous
considerations that $v$ satisfies (i) of Definition \ref{def3}, and then
$u$ satisfies (ii) of Definition \ref{def3}. So $u\in\Sch_\omega(\R^2)$, and
$L$ is $\omega$-regular for every weight $\omega$.

Looking at $L$, or equivalently at $\hat{L}$, as transformed operator
$\tilde{B}$
(of the form \eqref{H-O-2} with $Q\equiv R\equiv0$) we can apply
Theorem~\ref{th5} to obtain that $B$, defined by
\eqref{Harm-Os}, is $\omega$-regular for every $\omega$.
Applying again Theorem~\ref{th5} we have that
\eqref{H-O-1}, and in particular \eqref{H-O-2} and \eqref{H-O-3}, are
$\omega$-regular for every weight $\omega$.

Moreover, the harmonic oscillator $B_1(x,D_x)=x^2+D_x^2$ is
$\omega$-regular
for every $\omega$ also for
Proposition~\ref{propBxI}, from the $\omega$-regularity of
$B=B_1\hat{\otimes}I$.}
\end{Ex}

\begin{Rem}{\rm
\label{remex54}
Note that the symbol $(\xi-\frac12y)^2+(\eta+\frac12x)^2$ of
the twisted Laplacian $L$ defined in \eqref{twisted} is not hypoelliptic in the
sense of \eqref{g-hypoelliptic}, since it vanishes for $\xi=\frac12y$,
$\eta=-\frac12x$.}
\end{Rem}

\begin{Ex}{\rm
Another example comes from operators of the kind
$$
A(x,y,D_x,D_y)=D_x+\alpha x^m,
$$
for $\alpha\in\C$ and a positive integer $m$. The operator $A$ is regular in
Schwartz spaces for $(\Im\alpha)^m>0$, cf. \cite{BO}. Then Theorem \ref{th4}
gives us the regularity of
$$
\tilde{B}=\frac{D_x}{2}+y-P_2+\alpha\left(x-\frac{D_y}{2}-P_1\right)^m,
$$
for $P_1$ and $P_2$ as in Example \ref{ex-2}, cf.  \eqref{38} also.
In particular, if $P$ is of the form
$P(\xi,\eta)=\frac{1}{2}\xi\eta+P^{(1)}(\xi)$ or
$P(\xi,\eta)=-\frac{1}{2}\xi\eta+P^{(2)}(\eta)$ for polynomials
$P^{(1)}$ and $P^{(2)}$ with real coefficients, we obtain the regularity of
$$
y+\alpha\left( x-D_y+Q(D_x)\right)^m
$$
and
\beqsn
D_x+\alpha x^m+y+R(D_y)
\eeqsn
respectively, for a positive integer $m$, $\alpha\in \C$ satisfying
$(\Im\alpha)^m>0$, and for arbitrary differential operators $Q(D_x)$ and
$R(D_y)$ with constant real coefficients.}
\end{Ex}

{\bf Acknowledgments.}
The authors are grateful to Prof. Ernesto Buzano for useful discussions
concerning the examples.

The authors have been supported by GNAMPA-INdAM Project 2015 and by
FAR 2011 (University
of Ferrara). The second author was partially supported by MINECO, Project MTM2013-43540-P.

\end{document}